\documentclass[10pt,reqno]{amsart}

\usepackage{amssymb, amsmath, amsthm, amsfonts, mathtools}
\usepackage{enumitem}
\setlist[enumerate,1]{label=\tn{(\alph*)}}
\setlist[enumerate]{topsep=4pt, itemsep=3pt}
\usepackage{mathrsfs}
\usepackage{todonotes}
\usepackage{hyperref} 
\usepackage{xcolor}
\usepackage{tikz-cd}
\usepackage{physics}
\usepackage{bm}
\usepackage{comment}
\usepackage{dsfont}
\usepackage[makeroom]{cancel}
\usepackage{graphicx}
\newcommand{\downmapsto}{\rotatebox[origin=c]{-90}{$\scriptstyle\mapsto$}\mkern2mu}

\newcommand{\textdef}[1]{{\textit{#1}}}

\makeatletter
\g@addto@macro\bfseries{\boldmath}
\makeatother

\usepackage{hyperref}
  \definecolor{dark-red}{rgb}{0.6,0.15,0.15}
   \definecolor{dark-blue}{rgb}{0.15,0.15,0.6}
   \definecolor{medium-blue}{rgb}{0,0,0.5}

\hypersetup{
    colorlinks, 
    linkcolor=dark-red,
    citecolor=dark-blue, urlcolor=medium-blue
}

\numberwithin{equation}{section}

\theoremstyle{plain} 

\newtheorem{thm}{Theorem}[section]

\newtheorem{cor}[thm]{Corollary}

\newtheorem{prop}[thm]{Proposition}

\newtheorem{lem}[thm]{Lemma}
\newtheorem*{thm*}{Theorem}

\theoremstyle{definition}
\newtheorem{defn}[thm]{Definition}

\makeatletter
\let\c@equation=\c@equation
\let\c@lem=\c@thm
\let\c@cor=\c@thm
\let\c@conj=\c@thm
\let\c@prop=\c@thm
\let\c@lem=\c@thm
\let\c@defn=\c@thm
\let\c@notation=\c@thm
\let\c@note=\c@thm
\let\c@exmp=\c@thm
\let\c@ex=\c@thm
\let\c@exmps=\c@thm
\let\c@rem=\c@thm
\let\c@warn=\c@thm
\let\c@claim=\c@thm
\let\c@convention=\c@thm
\let\c@conventions=\c@thm
\let\c@quest=\c@thm
\let\c@facts=\c@thm
\makeatother

\DeclareMathOperator{\id}{id}

\DeclareMathOperator{\Aut}{\mathrm{Aut}}
\DeclareMathOperator{\Unitary}{\mathrm{U}}

\newcommand{\Z}{\mathbb{Z}}
\newcommand{\T}{\mathbb{T}}
\newcommand{\C}{\mathbb{C}}
\newcommand{\R}{\mathbb{R}}

\newcommand{\fA}{\mathfrak{A}}
\newcommand{\fB}{\mathfrak{B}}

\newcommand{\fN}{\mathfrak{N}}

\newcommand{\cH}{\mathcal{H}}

\newcommand{\cP}{\mathcal{P}}
\newcommand{\cE}{\mathcal{E}}

\renewcommand{\Im}{\operatorname{Im}}

\newcommand{\sS}{\mathscr{S}}
\newcommand{\sP}{\mathscr{P}}

\newcommand{\hilbH}{\mathcal{H}}

\DeclareMathAlphabet{\mathboondoxcal}{U}{BOONDOX-cal}{m}{n}
\SetMathAlphabet{\mathboondoxcal}{bold}{U}{BOONDOX-cal}{b}{n}
\DeclareMathAlphabet{\mathbboondoxcal} {U}{BOONDOX-cal}{b}{n}

\makeatletter
\DeclareFontFamily{OMX}{MnSymbolE}{}
\DeclareSymbolFont{MnLargeSymbols}{OMX}{MnSymbolE}{m}{n}
\SetSymbolFont{MnLargeSymbols}{bold}{OMX}{MnSymbolE}{b}{n}
\DeclareFontShape{OMX}{MnSymbolE}{m}{n}{
    <-6>  MnSymbolE5
   <6-7>  MnSymbolE6
   <7-8>  MnSymbolE7
   <8-9>  MnSymbolE8
   <9-10> MnSymbolE9
  <10-12> MnSymbolE10
  <12->   MnSymbolE12
}{}
\DeclareFontShape{OMX}{MnSymbolE}{b}{n}{
    <-6>  MnSymbolE-Bold5
   <6-7>  MnSymbolE-Bold6
   <7-8>  MnSymbolE-Bold7
   <8-9>  MnSymbolE-Bold8
   <9-10> MnSymbolE-Bold9
  <10-12> MnSymbolE-Bold10
  <12->   MnSymbolE-Bold12
}{}

\let\llangle\@undefined
\let\rrangle\@undefined
\DeclareMathDelimiter{\llangle}{\mathopen}%
                     {MnLargeSymbols}{'164}{MnLargeSymbols}{'164}
\DeclareMathDelimiter{\rrangle}{\mathclose}%
                     {MnLargeSymbols}{'171}{MnLargeSymbols}{'171}
\makeatother
\usepackage{mathtools}
\mathtoolsset{showonlyrefs}

\usepackage[normalem]{ulem}

\newcommand{\cK}{\mathcal{K}}
\newcommand{\cL}{\mathcal{L}}

\newcommand{\cU}{\mathcal{U}}

\newcommand{\bbS}{\mathbb{S}}
\newcommand{\bbP}{\mathbb{P}}
\newcommand{\bbR}{\mathbb{R}}
\newcommand{\bbZ}{\mathbb{Z}}
\newcommand{\bbN}{\mathbb{N}}
\newcommand{\tn}[1]{\textnormal{#1}}

\newcommand{\bbC}{\mathbb{C}}
\newcommand{\defeq}{\vcentcolon=}
\newcommand{\vecspan}{\tn{span}}
\DeclareMathOperator{\Log}{Log}

\newcommand{\B}{\mathcal{B}}
\newcommand{\1}{\mathds{1}}
\newcommand{\power}{\mathcal{P}}

\newcommand{\cKU}{\mathcal{U}}
\newcommand{\cKV}{\mathcal{V}}
\newcommand{\cKW}{\mathcal{W}}
\newcommand{\sQ}{\mathscr{Q}}
\DeclareMathOperator{\Ad}{Ad}
\newcommand{\cS}{\mathcal{S}}
\DeclareMathOperator{\diag}{diag}

\title{On the Weak Contractibility of the Space of Pure States}
\author[Spiegel]{Daniel D.\ Spiegel\textsuperscript{1,2}}
\author[Pflaum]{Markus J.\ Pflaum\textsuperscript{3,4}}

\thanks{\textsuperscript{1}Department of Mathematics, University of California, Davis}
\thanks{\textsuperscript{2}Center for Quantum Mathematics and Physics, University of California, Davis}
\thanks{\textsuperscript{3}Department of Mathematics, University of Colorado Boulder}
\thanks{\textsuperscript{4}Center for Theory of Quantum Matter, University of Colorado Boulder}

\begin{document}

\begin{abstract}
  We prove that the space $\sP(\fA)$ of pure states of a nonelementary, simple, separable, real rank zero $C^*$-algebra $\fA$ has trivial homotopy groups of all orders when $\sP(\fA)$ is equipped with the weak* topology. The convex-valued and finite-dimensional selection theorems of Michael are used to deform a family of pure states via the action of a homotopy of unitaries so that the entire family evaluates to one on a given projection $P \in \fA$. Then, the excision theorem of Akemann, Anderson, and Pedersen is used to iterate this deformation for a sequence of projections in $\fA$ excising a base point of the family of pure states, thereby contracting the family to the base point.
  Finally, we compare our weak contractibility result to the spaces of pure states of commutative $C^*$-algebras and rational rotation algebras, and compute the homotopy groups of the latter in terms of the homotopy groups of spheres.
\end{abstract}

\maketitle
\tableofcontents

\newpage
%!TEX root = pure_state_homotopy.tex

\section{Introduction}

This paper follows up on a prior result of the authors and their coauthors, namely that the pure state space $\sP(\fA)$ of an infinite-dimensional UHF algebra $\fA$ has trivial fundamental group when $\sP(\fA)$ is equipped with the weak* topology \cite[Thm.~4.21]{beaudry2023homotopical}. We note that the fundamental group of $\sP(\fA)$ is independent of the  base point because $\sP(\fA)$ is path-connected in the weak* topology, as follows from more general results of Eilers \cite{EilCCCA} (see also \cite[\S 1.4]{Spiegel} for a review). 

%In particular, it is immediate from the synthesis of Theorem 1.7, Theorem 5.6, and Proposition 5.9 of \cite{EilCCCA} that $\sP(\fA)$ is path-connected with respect to the weak* topology for every simple, separable $C^*$-algebra $\fA$.\footnote{The proof of \cite[Prop.~5.9]{EilCCCA} cites \cite[Thm.~3-17]{HockingYoung} for the fact that a connected and locally connected complete metric space is arcwise connected. The proof of this fact given in \cite{HockingYoung} is incorrect; this is pointed out in \cite{PersistenceofErrors}, wherein a correction to the proof is  provided.}

The main result of this paper is that \textit{all} homotopy groups of $\sP(\fA)$ are trivial with respect to the weak* topology. In fact, we prove this for all nonelementary, simple, separable $C^*$-algebras with real rank zero. While the intuition behind the broad idea of the proof is similar to the case of the fundamental group, the technical details are rather different. Thus, this paper provides an alternate (more complicated) proof of the result \cite[Thm.~4.21]{beaudry2023homotopical} for the fundamental group; it does not rely on the results of that paper.

The broad idea can even be demonstrated in the degree zero case, i.e., in how one might construct a path to connect an arbitrary pure state $\psi_0 \in \sP(\fA)$ to a fixed pure state $\omega \in \sP(\fA)$.  For now, let $\fA$ be a UHF algebra, let us think of $\fA$ as an infinite tensor product of matrix algebras $\fA \cong \bigotimes_{i \in \bbN} M_{d_i}(\bbC)$ for some numbers $d_i \geq 2$, and let us take $\omega$ to be the product state represented by the first standard basis vector on each tensor factor. Given $i \in \bbN$, define $E_i \in \fA$ to be the projection onto the first standard basis vector of $\bbC^{d_i}$, tensored with the identity on all other sites. We may apply a unitary $U \in \Unitary(\fA)$ to $\psi_0$ to obtain a new pure state $\psi_1(B) = \psi_0(U^*BU)$ such that $\psi_1(E_1) = 1$. The property $\psi_1(E_1) = 1$ implies that $\psi_1|_{M_{d_1}(\bbC)} = \omega|_{M_{d_1}(\bbC)}$.  Taking a path in $\Unitary(\fA)$ from the identity to $U$ induces a norm-continuous path in $\sP(\fA)$ from $\psi_0$ to $\psi_1$.  In the same manner we can find a path from $\psi_1$ to a pure state $\psi_2$ such that $\psi_2(E_1) = \psi_2(E_2) = 1$, hence $\psi_2|_{M_{d_1}(\bbC) \otimes M_{d_2}(\bbC)} = \omega|_{M_{d_1}(\bbC) \otimes M_{d_2}(\bbC)}$. Continuing this process yields, for all $i \in \bbN$, a pure state $\psi_i$  such that $\psi_i(E_j) = 1$ for all $j \leq i$ and a norm-continuous path from $\psi_{i-1}$ to $\psi_{i}$.

This countable infinity of paths can be fit inside the half-open interval $[0,1)$ by making the paths go progressively faster, for example with the ``speed'' of the paths increasing exponentially. This gives a norm-continuous function on $[0,1)$. At the endpoint $s = 1$ we define the path to be at $\omega$. We then have a weak*-continuous map on $I = [0,1]$ because, for any $B \in \fA$ that is the identity outside of finitely many tensor factors, evaluating the path of states on $B$ yields the constant $\omega(B)$ after some $s' \in [0,1)$.

In applying this procedure to higher degree homotopy groups, one is faced with the challenge of finding a homotopy from an arbitrary weak*-continuous family $\psi_0:\bbS^n \rightarrow \sP(\fA)$, $x \mapsto \psi_{0,x}$ to a family $\psi_1:\bbS^n \rightarrow \sP(\fA)$, $x \mapsto \psi_{1,x}$ satisfying $\psi_{1,x}(E_1) = 1$ for all $x \in \bbS^n$, where $\bbS^n$ is the $n$-sphere. In \cite{beaudry2023homotopical}, it was shown that this could be accomplished when $n = 1$ by acting on $\psi_0$ with a continuous family of operators supported in the first tensor factor. More precisely, one can construct a homotopy $B:\bbS^1 \times I \rightarrow M_{d_1}(\bbC)$ such that:
\begin{enumerate}
	\item $B(x,0)$ is the identity $\1$ for all $x \in \bbS^1$,
	\item $B(x,s) = \1$ whenever $\psi_{0,x}(E_1) = 1$,
	\item $\psi_{0,x}(B(x,s)^*B(x,s)) = 1$ for all $(x,s) \in \bbS^1 \times I$,
	\item $\psi_{0,x}(B(x,1)^*E_1B(x,1)) = 1$ for all $x \in \bbS^1$.
\end{enumerate}
Above we have identified $B(x,s)$ with its image under the canonical embedding $M_{d_1}(\bbC) \rightarrow \fA$. Acting with $B(x,s)$ on $\psi_{0,x}$ induces the desired homotopy in $\sP(\fA)$. Note that when $\sP(\fA)$ is given the weak* topology, we have continuity of the map
\[
\qty{(B, \psi) \in \fA \times \sP(\fA): \psi(B^*B) = 1} \rightarrow \sP(\fA), \quad (B, \psi) \mapsto B\psi
\] 
where we define
\begin{equation}\label{eq:algebra_action_state}
(B\psi)(C) \defeq \psi(B^*CB)
\end{equation}
for all $C \in \fA$. Thus, we may indeed act with $B(x,s)$ on $\psi_{0,x}$ to obtain a weak* continuous homotopy of $\psi_{0,x}$. One may then iterate this procedure as in the degree zero case to contract an arbitrary loop in $\sP(\fA)$.

For $n > 1$ such a homotopy $B$ cannot in general be found, at least not with $M_{d_1}(\bbC)$ as the codomain. For example, if $n = d_1 = 2$, then the pure state space of $M_{d_1}(\bbC)$ is homeomorphic to $\bbS^2$. Taking $\psi_0$ to be given by this homeomorphism, tensored with a fixed pure state on $\bigotimes_{i \geq 2} M_{d_i}(\bbC)$, we cannot find such a homotopy $B$ with codomain $M_{d_1}(\bbC)$ because this would induce a contraction of $\bbS^2$, which does not exist.

In this paper, we show how to find such a $B$ for arbitrary $n$ when the codomain is allowed to be all of $\fA$. In fact, we can find such a homotopy with image in the unitary group $\Unitary(\fA)$. This is supplied by Theorem \ref{thm:Michael_application} and Corollary \ref{cor:infinite_simple}, the latter of which holds for any infinite-dimensional, unital, simple $C^*$-algebra, where $E_1$ is replaced by an arbitrary projection $P \in \fA$. The iterative procedure for the degree zero case may then be carried through to contract an arbitrary $n$-sphere in $\sP(\fA)$, provided that the base point state is excised by a decreasing sequence of projections that evaluate to one on the base point state; this is done in Theorem \ref{thm:iteration_nulhomotopy}. This excision property holds for all states on a unital, separable, real rank zero $C^*$-algebra, as follows easily from the original excision theorem of Akemann, Anderson, and Pedersen \cite{AkemannAndersonPedersenExcision} and from results on real rank zero $C^*$-algebras by Brown and Pedersen \cite{BrownPedersenRealRankZero}.

The extension of the weak contractibility of spaces of pure states to the case of non-unital $C^*$-algebras is straightforward by unitization, except in the case where $\fA$ is infinite-dimensional and elementary, i.e., $*$-isomorphic to the algebra of compact operators on an infinite-dimensional Hilbert space $\cH$. In this case, $\sP(\fA)$ is an Eilenberg-MacLane space of type $K(\bbZ,2)$, as pointed out in Proposition \ref{prop:K(Z,2)}.

Contrary to the case of nonelementary, separable, simple $C^*$-algebras of  real rank zero, which are necessarily noncommutative,
commutative $C^*$-algebras show quite different behavior in regard to the homotopy types of their spaces of pure
states. Gelfand duality implies that all homotopy types of locally compact Hausdorff topological spaces can appear
in the commutative case. Noncommutative tori provide another rich class of examples. 
The irrational rotation algebras are nonelementary, separable, simple, and have real rank $0$, so our main result applies.
The case of rational rotation algebras is quite different, however. Their pure state spaces can be explicitly described
in terms of a certain bundle on the $2$-torus. This allows us to compute the homotopy groups of the pure state spaces
of rational rotation algebras in terms of the homotopy groups of spheres and shows that the pure state spaces in the rational case all have nontrivial homotopy type.
In particular, one concludes that these homotopy types are not invariant under deformation and are also not Morita invariant.

\begin{comment}
Since our homotopies have image in the unitary group, it is useful to make the following definition (cf. \cite[Def.~4.6]{beaudry2023homotopical}).

\begin{defn}
Let $\fA$ be a unital $C^*$-algebra, let $X$ be a topological space, and let $A \subset X$. A \textit{$\Unitary(\fA)$-homotopy relative to $A$} is a continuous map $U:X\times I \rightarrow \Unitary(\fA)$ such that
\begin{enumerate}
	\item $U(x,0) = \1$ for all $x \in X$,
	\item $U(x,s) = \1$ for all $x \in A$ and $s \in I$.
\end{enumerate}
We say two weak*-continuous maps $\psi:X \rightarrow \sP(\fA)$ and $\omega:X \rightarrow \sP(\fA)$ are \textit{$\Unitary(\fA)$-homotopic relative to $A$} if there exists a $\Unitary(\fA)$-homotopy relative to $A$ such that $\omega_x = U(x,1)\psi_x$ for all $x \in X$, where $U(x,1)\psi_x$ is defined as in \eqref{eq:algebra_action_state}. In this case we see that $(x,s) \mapsto U(x,s)\psi_x$ is a weak*-continuous homotopy from $\psi$ to $\omega$ relative to $A$ in the usual sense.
\end{defn}
\end{comment}

The key ingredients in many of our proofs come from the theory of continuous selections developed by Michael \cite{MichaelSelection,MichaelSelectionII}. We review this theory in Section \ref{sec:michael_selection}. In Section \ref{sec:continuous_unitaries} we apply Michael's convex-valued selection theorem to obtain a number of continuous families of unitaries in an arbitrary unital $C^*$-algebra. These continuous families serve as lemmas to later results but some of these constructions are interesting in their own right, e.g., as in Theorem \ref{thm:unitary_Kadison_lemma} and Theorem \ref{thm:weak*_selection_theorem}. In Section \ref{sec:equiLCn}, we find an equi-$LC^n$ family of subsets of $\Unitary(\fA)$, necessary for applying Michael's finite-dimensional selection theorem. In Section \ref{sec:michael_application} we apply the finite-dimensional selection theorem to prove our main results on the weak contractibility of the space of pure states.
Finally, in Section \ref{sec:comparison}, we consider the commutative case and the case of noncommutative tori.

We note that Michael's theory of selections has been fruitfully applied in the context of the homotopy theory of operator algebras before. Notably, Popa and Takesaki \cite{PopaTakesaki} used a theorem of Michael to prove contractibility of the automorphism group of the hyperfinite type II$_1$ factor $\mathfrak{R}_0$ in the topology induced by the predual. Building on this work, Andruchow and Varela showed that, for a fixed projection $P \in \mathfrak{R}_0$, the space of normal states of with support equivalent to $P$ is weakly contractible when endowed with the norm topology \cite{AndruchowVarela}, and the set of projections equivalent to $P$ is weakly contractible when endowed with the strong operator topology \cite{AndruchowContinuitySupportState}. Finally, we note that the homotopy theory of simple real rank zero $C^*$-algebras has been studied previously by Zhang, who computed the homotopy groups of the unitary group and the space of nontrivial projections for all nonelementary, simple $C^*$-algebras with real rank zero and stable rank one \cite{ZhangHomotopyRealRankZero}.

\subsection*{Notation and Conventions}

All Hilbert spaces will be over the complex numbers $\bbC$ and all inner products will be linear in the second variable. Given a Hilbert space $\hilbH$, we denote the unit sphere of $\cH$ as $\bbS \hilbH = \qty{x \in \hilbH: \norm{x} = 1}$. We denote the set of bounded linear operators by $\B(\hilbH)$ and the unitary group by $\Unitary(\hilbH)$. Likewise, if $\fA$ is a unital $C^*$-algebra, then we denote the unitary group of $\fA$ by $\Unitary(\fA)$. For any $C^*$-algebra $\fA$, we let $\sP(\fA)$ denote the set of pure states of $\fA$, equipped with the weak*-topology. If $\fA$ is unital, then we denote the unit by $\1$.

Given $\Psi,\Omega \in \hilbH$, we let $\ketbra{\Psi}{\Omega} \in \B(\hilbH)$ denote the operator defined by
\[
\ketbra{\Psi}{\Omega}\Phi = \ev{\Omega,\Phi}\Psi
\]
for all $\Phi \in \hilbH$.

We denote the unit interval by $I = [0,1]$. The power set of a set $X$ is denoted $\power(X)$. If $X$ is a metric space, $x \in X$, and $r > 0$, we let $B_r(x)$ denote the open ball of radius $r$ centered on $x$. 

\subsection*{Acknowledgements}
The authors would like to thank Agn\`es Beaudry and Bruno Nachtergaele for helpful conversations. In writing this paper, D.\ D.\ Spiegel was supported by the National Science Foundation (NSF) under Award No.\ DMS 2303063 and M.\ J.\ Pflaum was supported by the NSF under Award No.\ DMS 2055501.

%!TEX root = pure_state_homotopy.tex

\section{Michael's Theory of Selections}
\label{sec:michael_selection}

We review the key aspects of Michael's theory of continuous selections to be used in this paper. Throughout this section, let $X$ and $Y$ be topological spaces. Following \cite{MichaelSelection,MichaelSelectionII}, a function $\phi:X \rightarrow \cP(Y)\setminus \qty{\varnothing}$ will be called a \textdef{carrier}. One might think of $\phi(x)$ as being a set of solutions to a problem parametrized by $x$. Thus, for each $x$, one assumes that there exists at least one solution, but this solution might not be unique. Michael's results provide conditions under which a family of solutions may be chosen continuously in $x$, i.e., under which there exists a continuous function $s:X \rightarrow Y$ such that $s(x) \in \phi(x)$ for all $x \in X$. Such a continuous function is called a \textdef{selection} for $\phi$.

We will use two theorems of Michael in this paper, both of which require our carriers to be lower semicontinuous, as defined below.

\begin{defn}
A carrier $\phi:X \rightarrow \cP(Y)\setminus \qty{\varnothing}$ is  \textdef{lower semicontinuous} if for every $x_0 \in X$, $y_0 \in \phi(x_0)$, and neighborhood $O_{y_0}$ of $y_0$, there exists a neighborhood $O_{x_0}$ of $x_0$ such that $\phi(x) \cap O_{y_0} \neq \varnothing$ for every $x \in O_{x_0}$.
\end{defn}

With this we state below the first theorem of Michael to be used in this paper. This is sometimes known as the ``convex-valued selection theorem'' \cite{RepovsSemenovReview}.

\begin{thm}[{\cite[Thm.~3.2\textquotesingle\textquotesingle]{MichaelSelection}}]
If $X$ is paracompact Hausdorff, $Y$ is a (real or complex) Banach space, and $\phi:X \rightarrow \cP(Y)\setminus \qty{\varnothing}$ is a lower semicontinuous carrier such that $\phi(x)$ is closed and convex for all $x \in X$, then there exists a selection for $\phi$.
\end{thm}

We note that for \cite[Thm.~3.2\textquotesingle\textquotesingle]{MichaelSelection}, Michael also proves that if $X$ is a $T_1$ space for which selections exist for all carriers as described above, then $X$ is paracompact. We will not need this implication, however.

To state the second theorem, we first need another definition. We denote the $n$-sphere by $\bbS^n$.

\begin{defn}
Let $n \in \bbN \cup \qty{0}$. A collection of subsets $\cS \subset \cP(Y) \setminus \qty{\varnothing}$ is \textdef{equi-LC$^n$} at $y \in \bigcup \cS$ if for every open neighborhood $O$ of $y$, there exists an open neighborhood $y \in O' \subset O$ such that for every $m \leq n$, every $S \in \cS$, and every continuous function $f:\bbS^m \rightarrow S \cap O'$, the composition $\iota \circ f$ is nulhomotopic, where $\iota:S \cap O' \rightarrow S \cap O$ is the inclusion. By convention, every collection of neighborhoods $\cS \subset \cP(Y) \setminus \qty{\varnothing}$ is said to be equi-$LC^{-1}$.
\end{defn}

The second theorem of Michael we need can then be stated as follows. It is sometimes known as the ``finite-dimensional selection theorem'' \cite{RepovsSemenovReview}. 
Note that dimension in this context is always meant in the topological sense, that is $\dim X$ for a Hausdorff space $X$  
 stands for the Lebesgue covering dimension.  
We refer the reader to \cite{EngelkingDimensionTheory} for a general reference on dimension theory.

\begin{thm}[{\cite[Thm.~1.2]{MichaelSelectionII}}]\label{thm:finite-dim_Michael}
Let $n \in \bbN \cup \qty{-1, 0}$. Let $X$ be paracompact Hausdorff and let $A \subset X$ be a closed subset such that $\dim B \leq n + 1$ for every subset $B \subset X \setminus A$ that is closed in $X$.
%%%, where $\dim B$ is the Lebesgue covering dimension of $B$. 
If $Y$ is a complete metric space, $\cS \subset \cP(Y)\setminus \qty{\varnothing}$ is an equi-$LC^n$ family of closed subsets of $Y$, and $\phi:X \rightarrow \cP(Y) \setminus \qty{\varnothing}$ is a lower semicontinuous carrier with $\phi(x) \in \cS$ for all $x \in X$, then every selection for $\phi|_A$ can be extended to a selection for $\phi|_O$ for some open set $O$ containing $A$.
\end{thm}

In \cite[Thm.~1.2]{MichaelSelectionII}, Michael also provides conditions under which one can take $O = X$. These conditions will not be satisfied in our application, but we will nevertheless be able to take $O = X$ by other methods, to be shown.

In our application of the finite-dimensional selection theorem, $Y$ will be the unitary group of a unital $C^*$-algebra $\fA$. 
In this case, $Y$ is also a topological group when endowed with the norm topology. 
The following proposition simplifies the task of proving the equi-$LC^n$ condition for such $Y$. %%%in this case. 
We will use it in Theorem \ref{thm:equiLCn_pure_states}, where we show that a certain family $\cS$ of subsets of $Y = \Unitary(\fA)$ is equi-$LC^n$ at the identity $\1 \in \Unitary(\fA)$.

\begin{prop}\label{prop:equiLCn_topological_group}
Let $Y$ be a topological group and let $\cS$ be a nonempty collection of nonempty subsets of $Y$ such that 
\[
y \in Y \tn{ and }S \in \cS \,\, \Longrightarrow \,\, Sy^{-1} \in \cS,
\]
where $Sy^{-1} = \qty{xy^{-1}: x \in S}$. Then $\bigcup \cS = Y$ and, given $n \in \bbN \cup \qty{-1, 0}$, if $\cS$ is equi-$LC^n$ at a single $y_0 \in Y$, then $\cS$ is equi-$LC^n$.
\end{prop}

\begin{proof}
It is straightforward to see that $\bigcup \cS = Y$. Let $n \in \bbN \cup \qty{-1, 0}$ and suppose $\cS$ is equi-$LC^n$ at $y_0 \in Y$. We show $\cS$ is equi-$LC^n$ at an arbitrary $y \in Y$. If $n = -1$, then $\cS$ is trivially equi-$LC^n$, so suppose $n \geq 0$.

Let $O$ be a neighborhood of $y$. Then $O_0 = Oy^{-1}y_0$ is a neighborhood of $y_0$, so there exists a neighborhood $y_0 \subset O_0' \subset O_0$ such that for every $m \leq n$, every $S \in \cS$, and every continuous function $f_0:\bbS^m \rightarrow S \cap O_0'$, the composition $\iota_0 \circ f_0$ is nulhomotopic, where $\iota_0:S \cap O_0' \rightarrow S \cap O_0$ is the inclusion.

Let $O' = O_0'y_0^{-1}y$ and note that $O'$ is a neighborhood of $y$ contained in $O$. Given $m \leq n$, $S \in \cS$, and a continuous function $f:\bbS^m \rightarrow S \cap O'$, we define a new function $f_0:\bbS^m \rightarrow Sy^{-1}y_0 \cap O'_0$ by $f_0(x) = f(x)y^{-1}y_0$. Since $Sy^{-1}y_0 \in \cS$, we know there exists a homotopy $H_0:\bbS^m \times I \rightarrow Sy^{-1}y_0 \cap O_0$ such that $H_0(x, 0) = f_0(x)$ for all $x \in \bbS^m$ and $H_0(x, 1)$ is independent of $x$. Defining $H:\bbS^m \times I \rightarrow S \cap O$ by $H(x, t) = H_0(x,t)y_0^{-1}y$, we see that $H$ is a well-defined homotopy such that $H(x, 0) = f(x)$ for all $x \in \bbS^m$ and such that $H(x,1)$ is independent of $x$.
\end{proof}

%!TEX root = pure_state_homotopy.tex

\section{Continuous Families of Unitaries}
\label{sec:continuous_unitaries}

In this section we consider a unital $C^*$-algebra $\fA$ with a nonzero irreducible representation $(\cH, \pi)$ and show the existence of several useful continuous families of unitaries, characterized by their action on elements of $\cH$. The convex-valued selection theorem underlies the existence proofs in this section. % and
The norm bounds highlighted in the following results are often helpful in verifying the lower semicontinuity condition on our carriers.

First we define a unitary %that will be very useful.
which will be needed in the sequel. 

\begin{lem}\label{lem:U_xy}
Let $\hilbH$ be a Hilbert space. Define $\theta:\bbS \hilbH \times \bbS \hilbH \rightarrow [0,\pi]$ by
\begin{equation}\label{eq:theta_def}
\theta(\Psi,\Omega) = \cos^{-1}\qty(\Re \ev{\Psi,\Omega})
\end{equation}
Furthermore, given $\Psi, \Omega \in \bbS \hilbH$, let $\hilbH_{\Psi,\Omega} = \vecspan\qty{\Psi,\Omega}$ and define an operator $U_{\Psi,\Omega}:\hilbH_{\Psi,\Omega} \rightarrow \hilbH_{\Psi,\Omega}$ by
\begin{equation}\label{eq:U_xy}
U_{\Psi,\Omega} = \ev{\Omega,\Psi}\1 - \ketbra{\Psi}{\Omega} + \ketbra{\Omega}{\Psi}
\end{equation}
Then $U_{\Psi,\Omega}$ has the following properties:
\begin{enumerate}
	\item $U_{\Psi,\Omega}\Psi = \Omega$,
	\item $U_{\Psi,\Omega} \in \Unitary(\hilbH_{\Psi,\Omega})$,
	\item $\sigma(U_{\Psi,\Omega}) \subset \qty{e^{-i\theta(\Psi,\Omega)}, e^{i\theta(\Psi,\Omega)}}$,
	\item $\norm{\1 - U_{\Psi,\Omega}} = \norm{\Psi - \Omega}$.
\end{enumerate}
\end{lem}

\begin{proof}
First observe that if $\dim \hilbH_{\Psi,\Omega} = 1$, then $\Omega = \ev{\Psi, \Omega} \Psi$ and $\1 = \ketbra{\Psi}{\Psi}$, hence
\[
U_{\Psi,\Omega} = \ev{\Omega,\Psi}\1 - \ev{\Omega,\Psi}\ketbra{\Psi}{\Psi} + \ev{\Psi,\Omega}\ketbra{\Psi}{\Psi} = \ev{\Psi,\Omega}\1.
\]
If $\dim \hilbH_{\Psi,\Omega} = 2$, then we define
\begin{equation}\label{eq:w(x,y)}
\Phi(\Psi,\Omega) = \frac{\Omega - \ev{\Psi,\Omega}\Psi}{\norm{\Omega - \ev{\Psi,\Omega}\Psi}} = \frac{\Omega - \ev{\Psi,\Omega}\Psi}{\sqrt{1 - \abs{\ev{\Psi,\Omega}}^2}}
\end{equation}
and note that $\qty{\Phi(\Psi,\Omega), \Psi}$ is an orthonormal basis for $\hilbH_{\Psi,\Omega}$.

(a) We compute
\[
U_{\Psi,\Omega}\Psi = \ev{\Omega,\Psi}\Psi  - \ev{\Omega,\Psi}\Psi + \Omega = \Omega,
\]
as desired.

(b) If $\dim \hilbH_{\Psi,\Omega} = 1$, then $U_{\Psi,\Omega}$ is multiplication by $\ev{\Psi,\Omega}$, and $\abs{\ev{\Psi,\Omega}} = 1$, so $U_{\Psi,\Omega}$ is unitary. If $\dim \hilbH_{\Psi,\Omega} = 2$, then the matrix of $U_{\Psi,\Omega}$ in the basis $\qty{\Phi(\Psi,\Omega), \Psi}$ is
\begin{equation}\label{eq:U_xy_matrix}
U_{\Psi,\Omega} = \mqty(\ev{\Omega,\Psi} & \sqrt{1 - \abs{\ev{\Psi,\Omega}}^2} \\ -\sqrt{1 - \abs{\ev{\Psi,\Omega}}^2} & \ev{\Psi,\Omega}),
\end{equation}
which can be shown to be unitary by verifying $U_{\Psi,\Omega}^*U_{\Psi,\Omega} = \1$.

(c) Observe that
\[
e^{\pm i \theta(\Psi,\Omega)} = \Re \ev{\Psi,\Omega} \pm i \sqrt{1 - \qty(\Re \ev{\Psi,\Omega})^2}.
\]
If $\dim \hilbH_{\Psi,\Omega} = 1$, then $\abs{\ev{\Psi,\Omega}} = 1$, hence 
\[
e^{\pm i \theta(\Psi,\Omega)} = \Re \ev{\Psi, \Omega} \pm i \abs{\Im \ev{\Psi,\Omega}}.
\]
Since $\sigma(U_{\Psi,\Omega}) = \qty{\ev{\Psi,\Omega}}$ when $\dim \hilbH_{\Psi,\Omega} = 1$, we have $\sigma(U_{\Psi,\Omega}) \subset \qty{e^{-i\theta(\Psi,\Omega)}, e^{i\theta(\Psi,\Omega)}}$. If $\dim \hilbH_{\Psi,\Omega} = 2$, then diagonalizing the matrix \eqref{eq:U_xy_matrix} yields the result.

(d) By the spectral theorem,
\begin{align*}
\norm{\1 - U_{\Psi,\Omega}} &= \abs{1 - e^{\pm i\theta(\Psi,\Omega)}} = \sqrt{2 - 2 \Re \ev{\Psi,\Omega}} = \norm{\Psi - \Omega}. \qedhere
\end{align*}
\end{proof}

The logarithm of $U_{\Psi,\Omega}$ will also be useful.

\begin{lem}\label{lem:T_xy}
Assume the hypotheses of Lemma \ref{lem:U_xy} and assume $\ev{\Psi,\Omega} \in (0,1)$. Define $V_{\Psi,\Omega} = U_{\Psi,\Omega} \oplus \1$ with respect to the decomposition $\cH = \cH_{\Psi,\Omega} \oplus \cH_{\Psi,\Omega}^\perp$ and define $T_{\Psi,\Omega} = -i\Log V_{\Psi,\Omega}$, where the principal branch of the logarithm is used. Then 
\[
T_{\Psi,\Omega} = i\theta(\Psi,\Omega) \ketbra{\Psi}{\Phi(\Psi,\Omega)} - i\theta(\Psi,\Omega) \ketbra{\Phi(\Psi,\Omega)}{\Psi}
\]
where $\theta(\Psi, \Omega)$ and $\Phi(\Psi,\Omega)$ are defined by \eqref{eq:theta_def} and \eqref{eq:w(x,y)}, respectively.
\end{lem}

\begin{proof}
Note that $T_{\Psi,\Omega}$ is well-defined since $\sigma(V_{\Psi,\Omega}) \subset \qty{e^{-i\theta(\Psi,\Omega)}, e^{i\theta(\Psi,\Omega)}, 1}$ and this does not include $-1$ since $\ev{\Psi,\Omega} \in (0,1)$. It is clear that $T_{\Psi,\Omega}$ acts as the zero operator on $\cH_{\Psi,\Omega}^\perp$. One can easily check that $v^{\pm}_{\Psi,\Omega} = \Omega - e^{\mp i \theta(\Psi,\Omega)}\Psi$ are eigenvectors of $U_{\Psi,\Omega}$ with eigenvalues $e^{\pm i \theta(\Psi,\Omega)}$. Thus, $v^{\pm}_{\Psi,\Omega}$ are eigenvectors of $T_{\Psi,\Omega}$ with eigenvalues $\pm \theta(\Psi,\Omega)$. It is then elementary linear algebra to show that $T_{\Psi,\Omega}\Psi = -i\theta(\Psi,\Omega)\Phi(\Psi,\Omega)$ and $T_{\Psi,\Omega}\Phi(\Psi,\Omega) = i\theta(\Psi,\Omega)\Psi$.
\end{proof}

The following lemma was proven in \cite{Spiegel} and is used in the proof of Theorem \ref{thm:unitary_Kadison_lemma}. We reproduce it below for reference.

\begin{lem}[{\cite[Lem.~3.7]{Spiegel}}]\label{lem:Stiefel}
Let $\hilbH$ be a Hilbert space and let $\Psi_1,\ldots, \Psi_n \in \hilbH$ be an orthonormal system. Given $\varepsilon > 0$, there exists $\delta > 0$ such that for any orthonormal system $\Omega_1,\ldots, \Omega_n \in \hilbH$ with $\norm{\Psi_i - \Omega_i} < \delta$ for all $i$, there exists a unitary $U \in \Unitary(\hilbH)$ such that 
\begin{enumerate}
\item $\cK_n = \vecspan\qty{\Psi_1,\ldots, \Psi_n, \Omega_1,\ldots, \Omega_n}$ is invariant under $U$, 
\item $U$ acts as the identity on $\cK^\perp_n$, 
\item $\norm{I - U} < \varepsilon$, 
\item and $U\Psi_i = \Omega_i$ for all $i$.
\end{enumerate}
\end{lem}

A version of the following lemma was also proven in \cite{Spiegel}, although the norm bounds were not pointed out. We provide the full proof for the convenience of the reader. 

\begin{thm}\label{thm:unitary_Kadison_lemma}
Let $\fA$ be a unital $C^*$-algebra and let $(\hilbH, \pi)$ be a nonzero irreducible representation of $\fA$. Define
\[
Y_+ = \qty{(\Psi,\Omega) \in \bbS \hilbH \times \bbS \hilbH: \ev{\Psi, \Omega} \in (0,1]},
\]
equipped with the subspace topology. There exists a continuous map $\cKU: Y_+ \rightarrow \Unitary(\fA)$ such that
\begin{equation}\label{eq:U_transitive_lemma}
\pi(\cKU(\Psi,\Omega))\Psi = \Omega \qqtext{and} \norm{\1 - \cKU(\Psi,\Omega)} = \norm{\Psi - \Omega}
\end{equation}
for all $(\Psi, \Omega) \in Y_+$.
\end{thm}

\begin{proof}
We retain the notation from Lemma \ref{lem:U_xy}, Lemma \ref{lem:T_xy}, their proofs. Define a carrier $\phi:Y_+ \rightarrow \cP(\fA_\tn{sa})\setminus \qty{\varnothing}$ by
\[
\phi(\Psi,\Omega) = \qty{A \in \fA_\tn{sa}: \pi(A)P_{\Psi,\Omega} = T_{\Psi,\Omega}P_{\Psi,\Omega} \tn{ and } \norm{A} \leq \theta(\Psi,\Omega)}.
\]
By the Kadison transitivity theorem with norm bounds (see \cite[Thm.~3.4]{Spiegel} or \cite[Thm.~2.7.5]{PedersenCAlgAutomorphisms}), $\phi(\Psi,\Omega)$ is nonempty for all $(\Psi,\Omega) \in Y_+$, so $\phi$ is well-defined. It is easy to see that $\phi(\Psi,\Omega)$ is closed and convex for all $(\Psi,\Omega) \in Y_+$. Since $Y_+$ is metrizable, hence paracompact Hausdorff, we may apply the convex-valued selection theorem if we can show that $\phi$ is lower semicontinuous. 

Fix $(\Psi_0, \Omega_0) \in Y_+$, $A_0 \in \phi(\Psi_0, \Omega_0)$, and $\varepsilon > 0$. Suppose $\ev{\Psi_0, \Omega_0} = 1$. Then $\theta(\Psi_0, \Omega_0) = 0$, hence $A_0 = 0$. By continuity of $\theta$ on $Y_+$, we may choose a neighborhood $O$ of $(\Psi_0, \Omega_0)$ such that $(\Psi,\Omega) \in O$ implies $\theta(\Psi,\Omega) < \varepsilon$. Then for every $(\Psi,\Omega) \in O$ we have $\phi(\Psi,\Omega) \subset B_\varepsilon(A_0)$, so that in particular $\phi(\Psi,\Omega) \cap B_\varepsilon(A_0) \neq \varnothing$, as desired.

Suppose $\ev{\Psi_0, \Omega_0} < 1$. In this case $\theta(\Psi_0, \Omega_0)$ and $A_0$ are nonzero, the latter because $T_{\Psi_0,\Omega_0}P_{\Psi_0,\Omega_0}$ is nonzero. Using continuity of $\theta$ on $Y_+$, we can choose a neighborhood $O$ of $(\Psi_0, \Omega_0)$ such that for all $(\Psi,\Omega) \in O$ we have $\ev{\Psi,\Omega} < 1$ and 
\[
\abs{1 - \frac{\theta(\Psi,\Omega)}{\theta(\Psi_0, \Omega_0)}} < \frac{\varepsilon}{2\norm{A_0}}
\]
Applying Lemma \ref{lem:Stiefel} to the orthonormal system $\Psi_0, \Phi(\Psi_0, \Omega_0)$ and the number
\[
\varepsilon' = \min\qty(2, \frac{\varepsilon}{4\norm{A_0}}),
\]
we may find a $\delta > 0$ with the properties described in Lemma \ref{lem:Stiefel}. Note that $\Phi(\Psi,\Omega)$ is well-defined and continuous on $O$ since $\ev{\Psi,\Omega} < 1$ for $(\Psi,\Omega) \in O$. By continuity of $\Phi$, we may shrink $O$ such that for all $(\Psi,\Omega) \in O$, we have $\norm{\Psi - \Psi_0} < \delta$ and $\norm{\Phi(\Psi, \Omega) - \Phi(\Psi_0, \Omega_0)} < \delta$.

Now, given $(\Psi, \Omega) \in O$, there exists a unitary $V \in \Unitary(\hilbH)$ with the properties in Lemma \ref{lem:Stiefel}, in particular $V\Psi_0 = \Psi$ and $V\Phi(\Psi_0, \Omega_0) = \Phi(\Psi, \Omega)$, and $\norm{\1 - V} < \varepsilon'$. That $\norm{\1 - V} < 2$ implies that we can use the continuous functional calculus with the principal branch of the logarithm to define $S = -i\Log V \in \B(\hilbH)_\tn{sa}$, and $\norm{S} < \pi$. Note that $S$ leaves $\cK = \vecspan\qty{\Psi_0, \Phi(\Psi_0, \Omega_0), \Psi, \Phi(\Psi,\Omega)}$ invariant since $V$ and $V^*$ leave $\cK$ invariant. By the Kadison transitivity theorem with norm bounds, we obtain a self-adjoint operator $B \in \fA_\tn{sa}$ such that $\pi(B)|_{\cK} = S|_{\cK}$ and $\norm{B} \leq \norm{S}$. Hence, $W = e^{iB}$ acts as $V = e^{iS}$ on $\cK$ and by continuous functional calculus we know
\[
\norm{\1 - W} = \abs{1 - e^{i \norm{B}}} \leq \abs{1 - e^{i\norm{S}}} = \norm{\1 - V}.
\]
It is easy to check that
\[
A \defeq \frac{\theta(\Psi,\Omega)}{\theta(\Psi_0, \Omega_0)}WA_0W^{-1} \in \phi(\Psi,\Omega)
\]
using the values of $V$ and $T_{\Psi_0,\Omega_0}$ on $\Psi_0, \Phi(\Psi_0,\Omega_0)$ and the values of $V^{-1}$ and $T_{\Psi,\Omega}$ on $\Psi, \Phi(\Psi, \Omega)$. Finally, observe that
\begin{align*}
\norm{A_0 - A} &\leq \norm{A_0 - WA_0W^{-1}} + \abs{1 - \frac{\theta(\Psi,\Omega)}{\theta(\Psi_0, \Omega_0)}}\norm{A_0}\\
&\leq \qty(2 \norm{\1 - W} + \abs{1 - \frac{\theta(\Psi,\Omega)}{\theta(\Psi_0, \Omega_0)}})\norm{A_0} < \varepsilon,
\end{align*}
as desired. This proves lower semicontinuity of $\phi$.

By the convex-valued selection theorem we obtain a continuous selection $A:Y_+ \rightarrow \fA_\tn{sa}$ of $\phi$. Defining $\cU:Y_+ \rightarrow \Unitary(\fA)$ by $\cU(\Psi,\Omega) = e^{iA(\Psi,\Omega)}$, we see that $\cU$ is continuous,
\[
\pi(\cU(\Psi,\Omega)) = e^{i\pi(A(\Psi,\Omega))}\Psi = e^{iT_{\Psi,\Omega}}\Psi = \Omega,
\]
and
\[
\norm{\1 - \cU(\Psi,\Omega)} = \abs{1 - e^{i\norm{A(\Psi,\Omega)}}} \leq \abs{1 - e^{i\theta(\Psi,\Omega)}} = \norm{\Psi - \Omega}.
\]
for all $(\Psi,\Omega) \in Y_+$. We note that
\[
\norm{\1 - \cU(\Psi,\Omega)} \geq \norm{1 - \pi(\cU(\Psi,\Omega))} \geq \norm{\Psi - \pi(\cU(\Psi,\Omega))\Psi} = \norm{\Psi - \Omega},
\]
so equality holds in \eqref{eq:U_transitive_lemma}.
\end{proof}

The following corollary lifts the restriction that the inner product $\ev{\Psi, \Omega}$ be positive at the cost of a weaker norm bound.

\begin{cor}\label{cor:U_extension}
Let $\fA$ be a unital $C^*$-algebra and let $(\hilbH, \pi)$ be a nonzero irreducible representation of $\fA$. Define
\[
Y = \qty{(\Psi,\Omega) \in \bbS \cH \times \bbS \cH: \ev{\Psi,\Omega} \neq 0}
\]
For any continuous map $\cU:Y_+ \rightarrow \Unitary(\fA)$ satisfying \eqref{eq:U_transitive_lemma} for all $(\Psi,\Omega) \in Y_+$, there exists an extension of $\cU$ to a continuous map $\cU:Y \rightarrow \Unitary(\fA)$ such that
\begin{equation}\label{eq:U_extension}
\pi(\cU(\Psi,\Omega))\Psi = \Omega \qqtext{and} \norm{\1 - \cU(\Psi,\Omega)} \leq \norm{\Psi - \Omega} + \abs{1 - \frac{\ev{\Psi,\Omega}}{\abs{\ev{\Psi,\Omega}}}}
\end{equation}
for all $(\Psi,\Omega) \in Y$. In particular, there exists such a map $\cU:Y \rightarrow \Unitary(\fA)$.
\end{cor}

\begin{proof}
For $(\Psi,\Omega) \in Y$, we note that
\[
\qty(\Psi, \frac{\abs{\ev{\Psi,\Omega}}}{\ev{\Psi,\Omega}} \Omega) \in Y_+,
\]
so we can define $\cU:Y \rightarrow \Unitary(\fA)$ by
\[
\cU(\Psi,\Omega) = \frac{\ev{\Psi,\Omega}}{\abs{\ev{\Psi,\Omega}}}\cU\qty(\Psi,\frac{\abs{\ev{\Psi,\Omega}}}{\ev{\Psi,\Omega}} \Omega).
\]
It is straightforward to verify that this extends $\cU$ and satisfies \eqref{eq:U_extension}.
\end{proof}

Corollary \ref{cor:U_extension} can be used to prove the existence of a number of useful fiber bundles. The one we need here is given in the corollary below. Related results are given in \cite[\S3.2]{Spiegel}.

\begin{cor}\label{cor:principal_bundle}
Let $\fA$ be a unital $C^*$-algebra and let $(\cH, \pi)$ be a nonzero irreducible representation of $\fA$. Given $\Psi \in \bbS\cH$, the map
\[
\Unitary(\fA) \rightarrow \bbS \cH, \quad U \mapsto \pi(U)\Psi
\]
is a principal $\Unitary(\fA)_{\Psi}$-bundle, where
\[
U(\fA)_{\Psi} = \qty{U \in \Unitary(\fA): \pi(U)\Psi = \Psi}.
\]
\end{cor}

\begin{proof}
Since $\Unitary(\fA)$ acts transitively on $\bbS \cH$ by the Kadison transitivity theorem, it suffices to show that $U \mapsto \pi(U)\Psi$ has a continuous local section on a neighborhood of $\Psi$ \cite{Steenrod}. Let $O = \qty{\Omega \in \bbS \cH: \ev{\Psi,\Omega} \neq 0}$. With $\cU:Y \rightarrow \Unitary(\fA)$ as in Corollary \ref{cor:U_extension}, we see that 
\[
\sigma:O \rightarrow \Unitary(\fA), \quad \sigma(\Omega) = \cU(\Psi,\Omega) 
\]
is such a section.
\begin{comment}
Define 
\[
\lambda:O \rightarrow \Unitary(1), \quad \lambda(\Omega) = \frac{\ev{\Omega_0, \Omega}}{\abs{\ev{\Omega_0, \Omega}}}
\]
Observe that $\ev{\Omega_0, \lambda(\Omega)^{-1}\Omega)} > 0$. Therefore we can define
\[
\sigma:O \rightarrow \Unitary(\fA), \quad \sigma(\Omega) = \lambda(\Omega) U\qty(\Omega_0, \lambda(\Omega)^{-1}\Omega),
\]
where $U:Y \rightarrow \Unitary(\fA)$ is as in Lemma \ref{lem:unitary_Kadison_lemma}. Continuity of $\sigma$ follows from continuity of $\lambda$ and $U$, and
\[
\pi(\sigma(\Omega))\Omega_0 = \Omega
\]
by \eqref{eq:U_transitive_lemma}.
\end{comment}
\end{proof}

One of our main results in Section \ref{sec:michael_application} will be to show that a family of pure states $\omega:X \rightarrow \sP(\fA)$ can be deformed by a homotopy of unitaries so as to all evaluate to one on a given projection $P \in \fA$. The following theorem is a precursor to that result, where we allow ourselves the additional assumption that $\omega_x(P) > 0$ for all $x \in X$.

\begin{thm}\label{thm:weak*_selection_theorem}
Let $\fA$ be a unital $C^*$-algebra and let $P \in \fA$ be a projection. Let $X$ be a paracompact Hausdorff space and let $\omega:X \rightarrow \sP(\fA)$, $x \mapsto \omega_x$ be weak*-continuous and satisfy $\omega_x(P) > 0$ for all $x \in X$. There exists a continuous map $U:X \rightarrow \Unitary(\fA)$ such that 
\[
\omega_x(U_x^*PU_x) = 1
\]
for all $x \in X$ and 
\[
\norm{\1 - U_x} = \sqrt{2 - 2\sqrt{\omega_x(P)}}.
\]
\end{thm}

\begin{proof}
First, let us establish some notation. Given $x \in X$, let $(\cH_x, \pi_x, \Omega_x)$ be the GNS representation of $\omega_x$. Let
\begin{equation}\label{eq:Psi_x_def}
\Psi_x = \frac{\pi_x(P)\Omega_x}{\norm{\pi_x(P)\Omega_x}} = \frac{\pi_x(P)\Omega_x}{\sqrt{\omega_x(P)}}
\end{equation}
and let 
\begin{equation}\label{eq:Upsilon_x_def}
\Upsilon_x = \Psi_x - \ev{\Omega_x, \Psi_x}\Omega_x.
\end{equation}
We compute:
\begin{align*}
\ev{\Omega_x, \Psi_x} &= \sqrt{\omega_x(P)}, \\
 \norm{\Omega_x - \Psi_x} &= \sqrt{2 - 2\sqrt{\omega_x(P)}}, \\
 \qqtext{and} \norm{\Upsilon_x} &= \sqrt{1 - \omega_x(P)}
\end{align*}
and note that these are all continuous functions of $x$.

Let $\cH_{\Omega_x, \Psi_x} = \vecspan\qty{\Omega_x, \Psi_x}$ and let $V_{\Omega_x, \Psi_x} \in \Unitary(\cH_x)$ and $T_{\Omega_x, \Psi_x} \in \B(\cH_x)_\tn{sa}$ be as in Lemma \ref{lem:T_xy}. Note that $\norm{T_{\Omega_x, \Psi_x}} = \theta(\Omega_x, \Psi_x)$, where $\theta(\Omega_x, \Psi_x)$ is as in \eqref{eq:theta_def}. Let $Q_x \in \B(\cH_x)$ be the projection onto $\cH_{\Omega_x, \Psi_x}$. Consider the carrier $\phi:X \rightarrow \fA_\tn{sa}$ defined by 
\[
\phi(x) = \qty{A \in \fA_\tn{sa}: \pi_x(A)Q_x = T_{\Omega_x, \Psi_x}Q_x \tn{ and } \norm{A} \leq \norm{T_{\Omega_x, \Psi_x}}}.
\]
By the Kadison transitivity theorem with norm bounds (see \cite[Thm.~3.4]{Spiegel} or \cite[Thm.~2.7.5]{PedersenCAlgAutomorphisms}), $\phi(x)$ is nonempty for all $x \in X$. It is clear that $\phi(x)$ is closed and convex for all $x \in X$. Since $X$ is paracompact Hausdorff, the convex-valued selection theorem will apply if we can show that $\phi$ is lower semicontinuous. 

Fix $x_0 \in X$, $A_0 \in \phi(x_0)$, and $\varepsilon > 0$. If $\omega_{x_0}(P) = 1$, then $A_0 = 0$. Observe that $x \mapsto \theta(\Omega_x, \Psi_x) = \cos^{-1}\sqrt{\omega_x(P)}$ is continuous, therefore we may choose an open neighborhood $O$ of $x_0$ such that $\theta(\Omega_x,\Psi_x) < \varepsilon$ for all $x \in O$. If $x \in O$, then $\phi(x) \subset B_\varepsilon(A_0)$, so in particular $\phi(x) \cap B_\varepsilon(A_0) \neq \varnothing$, as desired.

Now consider the case where $\omega_{x_0}(P) < 1$. We first set $O = \qty{x \in X: \omega_x(P) < 1}$, which is an open neighborhood of $x_0$. On $O$, we know $\Psi_x$ and $\Omega_x$ are linearly independent, hence $\Upsilon_x \neq 0$.  Using the form of $T_{\Omega_x, \Psi_x}$ given in Lemma \ref{lem:T_xy}, one has
\begin{equation}\label{eq:big_ugly_1}
\norm{T_{\Omega_x, \Psi_x}\Omega_x - \pi_x(A_0)\Omega_x} = \norm{-i\theta(\Omega_x, \Psi_x)\frac{\Upsilon_x}{\norm{\Upsilon_x}}   - \pi_x(A_0)\Omega_x}
\end{equation}
and
\begin{equation}\label{eq:big_ugly_2}
\norm{T_{\Omega_x, \Psi_x}\Upsilon_x - \pi_x(A_0)\Upsilon_x} = \norm{i\theta(\Omega_x, \Psi_x)\norm{\Upsilon_x}\Omega_x - \pi_x(A_0)\Upsilon_x}.
\end{equation}
Expressing $\Upsilon_x$ in terms of $\Omega_x$ using \eqref{eq:Upsilon_x_def} and \eqref{eq:Psi_x_def}, one can see that both \eqref{eq:big_ugly_1} and \eqref{eq:big_ugly_2} are of the form $x \mapsto \norm{\pi_x(B_x)\Omega_x} = \sqrt{\omega_x(B_x^*B_x)}$ for a continuous function $B: O \rightarrow \fA$, $x \mapsto B_x$. Functions of this form are continuous by weak*-continuity of $\omega$, so \eqref{eq:big_ugly_1} and \eqref{eq:big_ugly_2} are continuous functions of $x$ on $O$. Since $\eqref{eq:big_ugly_1}$ and $\eqref{eq:big_ugly_2}$ evaluate to zero at $x_0$, we may shrink $O$ so that these norms are less than $\varepsilon/6$ for all $x \in O$. Let us also shrink $O$ so that 
\[
\abs{\theta(\Omega_x, \Psi_x) - \theta(\Omega_{x_0}, \Psi_{x_0})} < \frac{\varepsilon}{3}
\]
for all $x \in O$.

Then for any $x \in O$,  there exists a self-adjoint $S \in \B(\cH_x)_\tn{sa}$ such that
\[
S\Omega_x = T_{\Omega_x, \Psi_x}\Omega_x - \pi_x(A_0)\Omega_x \qqtext{and} S\Upsilon_x = T_{\Omega_x, \Psi_x}\Upsilon_x - \pi_x(A_0)\Upsilon_x
\]
and $\norm{S} < \varepsilon/3$, by \cite[Lem.\ 5.4.2]{KadisonRingroseI}. By the Kadison transitivity theorem with norm bounds, there exists $A_1 \in \fA_\tn{sa}$ such that
\[
\pi_x(A_1)\Omega_x = T_{\Omega_x, \Psi_x}\Omega_x - \pi_x(A_0)\Omega_x \qqtext{and} \pi_x(A_1)\Upsilon_x = T_{\Omega_x, \Psi_x}\Upsilon_x - \pi_x(A_0)\Upsilon_x
\]
and $\norm{A_1} < \varepsilon/3$. Observe that $A_0 + A_1 \in B_\varepsilon(A_0)$ and $\pi_x(A_0 + A_1)Q_x = T_{\Omega_x, \Psi_x}Q_x$, but we do not necessarily have the desired norm bound on $A_0 +  A_1$ to conclude that $A_0 + A_1 \in \phi(x)$.

To fix this, consider the continuous function $f:\bbR \rightarrow \bbR$ defined by
\[
f(\lambda) = \begin{cases} \theta(\Omega_x, \Psi_x)&: \lambda \geq \theta(\Omega_x, \Psi_x) \\ \lambda &: \abs{\lambda} \leq \theta(\Omega_x, \Psi_x) \\ -\theta(\Omega_x, \Psi_x) &: \lambda \leq - \theta(\Omega_x, \Psi_x)\end{cases}.
\]
Then $f(A_0 + A_1)$ is self-adjoint and $\norm{f(A_0 + A_1)} \leq \theta(\Omega_x, \Psi_x)$. Furthermore, since $\norm{A_0 + A_1} < \theta(\Omega_{x_0}, \Psi_{x_0}) + \varepsilon/3$, we know 
\[
\abs{\lambda - f(\lambda)} \leq \abs{\theta(\Omega_{x_0}, \Psi_{x_0}) +  \frac{\varepsilon}{3} - \theta(\Omega_x, \Psi_x)} < \frac{2\varepsilon}{3}
\]
for all $\lambda \in \sigma(A_0 + A_1)$. Thus,
\begin{align*}
\norm{A_0 - f(A_0 + A_1)} &\leq \norm{A_1} + \norm{A_0 + A_1 - f(A_0 + A_1)} < \varepsilon, 
\end{align*}

Finally, observe that $\Omega_x \mp i\norm{\Upsilon_x}^{-1} \Upsilon_x$ are eigenvectors of $T_{\Omega_x, \Psi_x}$ with eigenvalues $\pm \theta(\Omega_x, \Psi_x)$. Since $\pi_x(A_0 + A_1)Q_x = T_{\Omega_x, \Psi_x}Q_x$, these are also eigenvectors of $\pi_x(A_0 + A_1)$, with the same eigenvalues. Hence, they are also eigenvectors of $\pi_x(f(A_0 + A_1))$ with eigenvalues $f(\pm \theta(\Omega_x, \Psi_x)) = \pm \theta(\Omega_x, \Psi_x)$. Since these eigenvectors span $\cH_{\Omega_x, \Psi_x}$, we see that $\pi_x(f(A_0 + A_1))Q_x = T_{\Omega_x, \Psi_x}Q_x$. Therefore $A_0 + A_1 \in \phi(x) \cap B_\varepsilon(A_0)$. This proves lower semicontinuity of $\phi$.

Therefore, by the convex-valued selection theorem, there exists a continuous selection $A:X\rightarrow \fA_\tn{sa}$, $x \mapsto A_x$ for $\phi$. Then $U : X \rightarrow \Unitary(\fA)$, $x \mapsto e^{iA_x}$ is continuous and satisfies 
\[
\pi_x(U_x)\Omega_x = e^{i\pi_x(A_x)}\Omega_x = e^{i\pi_x(A_x)}\Omega_x = e^{iT_{\Omega_x, \Psi_x}}\Omega_x = V_{\Omega_x,\Psi_x}\Omega_x = \Psi_x,
\]
from which it follows that $\omega_x(U_x^*PU_x) = 1$. Since $\norm{A_x} \leq \theta(\Omega_x, \Psi_x) < \pi/2$ for all $x \in X$, we also have
\[
\norm{\1 - U_x} = \abs{1 - e^{i\norm{A_x}}} \leq \abs{1 - e^{i\theta(\Omega_x, \Psi_x)}}  = \sqrt{2 - 2\sqrt{\omega_x(P)}},
\]
as desired.
\end{proof}

\begin{cor}\label{cor:unitary_onto_projection}
Let $\fA$ be a unital $C^*$-algebra, let $(\hilbH, \pi)$ be a nonzero irreducible representation of $\fA$, and let $P \in \B(\hilbH)$ be a nonzero projection. Let 
\[
X = \qty{\Psi \in \bbS \hilbH: P\Psi \neq 0}.
\]
There exists a continuous map $U:X \rightarrow \Unitary(\fA)$, $\Psi \mapsto U_\Psi$ such that
\[
\pi(U_\Psi)\Psi = \frac{P\Psi}{\norm{P\Psi}}
\]
and 
\[
\norm{\1 - U_\Psi} = \norm{\frac{P\Psi}{\norm{P\Psi}} - \Psi} = \sqrt{2 - 2\norm{P\Psi}}.
\]
for all $\Psi \in X$.
\end{cor}

\begin{proof}
Let $Y_+$ be as in Theorem \ref{thm:unitary_Kadison_lemma}. Observe that we have a continuous map $X \rightarrow Y_+$, $\Psi \mapsto (\Psi, P\Psi/\norm{P\Psi})$. Composing this map with the continuous map $Y_+ \rightarrow \Unitary(\fA)$ from Theorem \ref{thm:unitary_Kadison_lemma} yields the desired result.
\end{proof}

The families of unitaries developed in the lemma below and in Lemma \ref{lem:ZPgammadelta} will be used in Section \ref{sec:equiLCn} to show that a certain collection of subsets of $\Unitary(\fA)$ is equi-$LC^n$, as is needed for applying the finite-dimensional selection theorem.

\begin{lem}\label{lem:V_deform_to_U}
Let $\fA$ be a unital $C^*$-algebra, let $(\hilbH, \pi)$ be a nonzero irreducible representation of $\fA$, and let $\Psi \in \bbS \hilbH$. Let
\[
Z = \qty{U \in \Unitary(\fA): \ev{\Psi, \pi(U)\Psi} > 0 \tn{ and } \norm{\1 - U} < 1}.
\]
Let $Y_+$ and $\cKU:Y_+ \rightarrow \Unitary(\fA)$ be as in Theorem \ref{thm:unitary_Kadison_lemma}. There exists a continuous function $\cKV:Z \times I \rightarrow \Unitary(\fA)$ such that 
\begin{enumerate}
	\item\label{ite:V(U,0)} $\cKV(U,0) = U$, 
	\item\label{ite:V(U,1)} $\cKV(U,1) = \cKU(\Psi, \pi(U)\Psi)$,
	\item\label{ite:V(U,s)x} $\pi(\cKV(U,s))\Psi = \pi(U)\Psi$,
	\item\label{ite:V(U,s)_norm} $\norm{\1 - \cKV(U,s)} \leq 3\norm{\1 - U}$
\end{enumerate}
for all $U \in Z$ and $s \in I$.
\end{lem}

\begin{proof}
Define $W:Z \rightarrow \Unitary(\fA)$ by 
\[
W(U) = \cKU(\Psi, \pi(U)\Psi)^* U.
\]
Then $W$ is manifestly continuous and 
\begin{align*}
\norm{\1 - W(U)} &\leq \norm{\1 - \cKU(\Psi, \pi(U)\Psi)^*} + \norm{\1 - U} \\
&= \norm{\Psi -\pi(U)\Psi} + \norm{\1 - U} \\
&\leq 2\norm{\1 - U} < 2.
\end{align*}
Since $\norm{\1 - W(U)} < 2$, we can apply the principal branch of the logarithm to obtain a continuous function $A:Z \rightarrow \fA_\tn{sa}$, $A(U) = -i\Log W(U)$ with $\norm{A(U)} < \pi$ for all $U \in Z$. 

Observe that for any $U \in Z$ we have
\[
\pi(W(U))\Psi = \pi(\cKU(\Psi, \pi(U)\Psi))^{-1} \pi(U)\Psi = \Psi.
\]
It follows from continuous functional calculus that $\pi(A(U))\Psi = 0$.

We now define $\cKV:Z \times I \rightarrow \Unitary(\fA)$ by
\begin{align*}
\cKV(U,s) = \cKU(\Psi, \pi(U)\Psi)e^{i(1-s)A(U)}.
\end{align*}
At $s = 0$, we have
\[
\cKV(U, 0) = \cKU(\Psi, \pi(U)\Psi)W(U) = U.
\]
At $s = 1$, we have
\[
\cKV(U,1) = \cKU(\Psi, \pi(U)\Psi).
\]
For any $(U,s) \in Z \times I$, we have
\begin{align*}
\pi(\cKV(U,s))\Psi &= \pi(\cKU(\Psi, \pi(U)\Psi))e^{i(1 - s)\pi(A(U))}\Psi  \\
&= \pi(\cKU(\Psi, \pi(U)\Psi))\Psi = \pi(U)\Psi,
\end{align*}
and finally
\begin{align*}
\norm{\1 - \cKV(U,s)} &\leq \norm{\1 - \cKU(\Psi, \pi(U)\Psi)} + \norm{\1 - e^{i(1 - s)A(U)}}\\
&\leq \norm{\1 - U} + \norm{\1 - W(U)} \leq 3\norm{\1 - U}.
\end{align*}
This completes the proof.
\end{proof}

Our final continuous map in $\Unitary(\fA)$ involves compressing the representation $(\cH, \pi)$. We show that the compression of an irreducible representation is irreducible. This is an elementary fact, but we supply the proof for the convenience of the reader.

\begin{lem}\label{lem:compressed_rep}
  Let $\fA$ be a $C^*$-algebra, $P \in \fA$ a nonzero projection, and let
  $\pi:\fA \rightarrow \B(\hilbH)$  a representation of $\fA$. Denote by 
  $p:\hilbH \rightarrow \pi(P)\hilbH$ the orthogonal projection onto $\pi(P)\hilbH$ and define 
\[
\pi_P:P\fA P \rightarrow \B(\pi(P)\hilbH), \quad \pi_P(A) = p\pi(A)p^* \ .
\]
Note that the adjoint $p^*:\pi(P)\hilbH \rightarrow \hilbH$ is the inclusion of $\pi(P)\hilbH$ into $\hilbH$. Then the following holds true.
\begin{enumerate}
\item
  The map $\pi_P$ is a representation of $P\fA P$.
\item
  If $\pi(P)$ is nonzero and $\pi$ is irreducible,  then $\pi_P$ is nonzero and irreducible.
\item\label{ite:c}
  Assume that $\omega \in \sP(\fA)$ and that $(\cH, \pi)$ is the GNS representation of $\omega$.
  If $\omega(P) = 1$, then the restriction of $\omega$ to $P\fA P$ is pure and the GNS representation %$(\cH_{\omega|_{P\fA P}}, \pi_{\omega|_{P\fA P}})$
  of $\omega|_{P\fA P}$ is unitarily equivalent to $(\pi(P)\cH, \pi_P)$,
\end{enumerate}
\end{lem}

\begin{proof}
  (a) It is clear that $\pi_P$ is linear and preserves the $*$-operation. For any $A \in P\fA P$, we have
  $AP = PA = A$. Since $p^*p = \pi(P)$, we have for any $A, B \in P\fA P$,
\begin{align*}
\pi_P(AB) = p\pi(A)\pi(B)p^* = p\pi(A)\pi(P)\pi(B)p^* = p\pi(A)p^* p \pi(B)p^* = \pi_P(A) \pi_P(B),
\end{align*}
so $\pi_P$ is a representation.

(b) Suppose $\pi(P)$ is nonzero and $\pi$ is irreducible. Then $\pi_P(P) = pp^*pp^* = \id_{\pi(P)\hilbH}$, so
$\pi_P$ is nonzero. To show that $\pi_P$ is irreducible, let $\Psi, \Omega \in \pi(P)\hilbH$ be arbitrary
nonzero vectors. By the Kadison transitivity theorem, there exists $A \in \fA$ such that $\pi(A)\Psi = \Omega$.
Thus,
\[
  \pi_P(PAP)\Psi = p\pi(PAP)p^* \Psi = p\pi(A)p^* \Psi = p\pi(A)\Psi = p\Omega = \Omega.
\]
It follows that if $\cK$ is a nonzero invariant subspace for $\pi_P$, then $\cK = \pi(P)\hilbH$. Thus,
$\pi_P$ is irreducible.

(c) Let $\Omega \in \cH$ be a cyclic unit vector representing $\omega$. Since $\omega(P) = 1$, we know
$\pi(P)\Omega = \Omega$. By (a) and (b) one concludes that $(\pi(P)\cH, \pi_P)$ is an irreducible representation of
$P\fA P$ with $\Omega \in \pi(P)\cH$ representing $\omega|_{P\fA P}$. It follows that the GNS representation of $\omega|_{P\fA P}$ is unitarily equivalent to $(\pi(P)\cH, \pi_P)$, hence $\omega|_{P\fA P}$ is pure since its GNS representation is irreducible.
\end{proof}

The following elementary bound will be used several times in the rest of the paper. Again, we provide the proof for the convenience of the reader.

\begin{lem}\label{lem:normalization_bound}
If $\cH$ is a Hilbert space and $\Psi, \Omega \in \cH \setminus \qty{0}$, then
\begin{equation}\label{eq:normalization_bound}
\norm{\frac{\Psi}{\norm{\Psi}} - \frac{\Omega}{\norm{\Omega}}}^2 \leq \frac{\norm{\Psi - \Omega}^2}{\norm{\Psi}\norm{\Omega}}
\end{equation}
with equality if and only if $\norm{\Psi} = \norm{\Omega}$.
\end{lem}

\begin{proof}
We compute:
\begin{align*}
\norm{\frac{\Psi}{\norm{\Psi}} - \frac{\Omega}{\norm{\Omega}}}^2 &= 2 - 2\Re \qty(\frac{\ev{\Psi,\Omega}}{\norm{\Psi}\norm{\Omega}}) =  2 + \frac{\norm{\Psi - \Omega}^2 - \norm{\Psi}^2 - \norm{\Omega}^2}{\norm{\Psi}\norm{\Omega}}
\end{align*}
Using the inequality 
\begin{equation}\label{eq:quadratic_ineq}
-\norm{\Psi}^2 - \norm{\Omega}^2 \leq -2\norm{\Psi}\norm{\Omega}
\end{equation}
yields \eqref{eq:normalization_bound}. Since equality holds in \eqref{eq:quadratic_ineq} if and only if $\norm{\Psi} = \norm{\Omega}$, we see that equality holds in \eqref{eq:normalization_bound} if and only if $\norm{\Psi} = \norm{\Omega}$. 
\end{proof}

\begin{lem}\label{lem:ZPgammadelta}
Let $\fA$ be a unital $C^*$-algebra, let $(\cH, \pi)$ be a nonzero irreducible representation of $\fA$, let $P \in \fA$ be a projection, and let $\Psi \in \bbS\cH$ such that $\pi(P)\Psi \neq 0$. 
%Given $U \in \Unitary(\fA)$ such that $\pi(PU)\Psi \neq 0$, define
%\[
%y_{P,U} = \frac{\pi(PU)\Psi}{\norm{\pi(PU)\Psi}}.
%\]
Let $\gamma, \delta > 0$ and suppose
\begin{equation}\label{eq:beta_gamma_delta_condition}
\delta^2 < 2\gamma \norm{\pi(P)\Psi} \qqtext{and} \beta \delta < 2,
\end{equation}
where
\begin{equation}\label{eq:beta_def}
\beta = \frac{1}{\sqrt{\gamma \norm{\pi(P)\Psi}}} + \sqrt{\frac{2}{2\gamma \norm{\pi(P)\Psi} - \delta^2}}.
\end{equation}
Define
\[
Z_{P,\gamma,\delta} = \qty{U \in \Unitary(\fA): \norm{\pi(PU)\Psi} \geq \gamma \tn{ and } \norm{\1 - U} < \delta}
\]
Then there exists a continuous function $\cKW :Z_{P,\gamma, \delta} \times I \rightarrow \Unitary(\fA)$ such that, for all $U \in Z_{P,\gamma, \delta}$ and $s \in I$:
\begin{enumerate}
	\item\label{ite:W(U,0)} $\cKW(U, 0) = U$,
	\item\label{ite:(1-P)W(U,s)} $(\1 - P)\cKW(U,s) = (\1 - P)U$,
	\item\label{ite:PW(U,s)x} $\norm{\pi(P\cKW(U,s))\Psi} = \norm{\pi(PU)\Psi}$,
	\item\label{ite:PW(U,1)x} $\pi(P\cKW(U, 1))\Psi = \norm{\pi(PU)\Psi}\cdot\dfrac{\pi(P)\Psi}{\norm{\pi(P)\Psi}}$,
	\item\label{ite:W(u,s)_norm} $\norm{\1 - \cKW(U,s)} \leq (1+\beta) \norm{\1 - U}$,
\end{enumerate}
\end{lem}

\begin{proof}
For any $U \in \Unitary(\fA)$ such that $\pi(PU)\Psi \neq 0$, define
\[
\Omega_U = \frac{\pi(PU)\Psi}{\norm{\pi(PU)\Psi}}.
\]
Let us first compute a few  estimates. For any $U \in Z_{P,\gamma,\delta}$, Lemma \ref{lem:normalization_bound} yields
\begin{equation}\label{eq:y_U-y_1}
\norm{\Omega_U - \Omega_\1}^2 \leq \frac{\norm{\pi(PU)\Psi - \pi(P)\Psi}^2}{\norm{\pi(PU)\Psi}\norm{\pi(P)\Psi}} \leq \frac{\norm{\1 - U}^2}{\gamma \norm{\pi(P)\Psi}}.
\end{equation}
Next, for any $U \in Z_{P,\gamma,\delta}$, we estimate
\begin{equation}\label{eq:|y_U,y_1|}
\abs{\ev{\Omega_U, \Omega_\1}} \geq \Re \ev{\Omega_U, \Omega_\1} = 1 - \frac{1}{2}\norm{\Omega_U - \Omega_\1}^2 > 1 - \frac{\delta^2}{2\gamma \norm{\pi(P)\Psi}} > 0.
\end{equation}
Finally, for any $U \in Z_{P,\gamma,\delta}$, Lemma \ref{lem:normalization_bound} once again yields
\begin{equation}
\abs{1 - \frac{\ev{\Omega_U, \Omega_\1}}{\abs{\ev{\Omega_U, \Omega_\1}}}}^2 \leq \frac{\abs{1 - \ev{\Omega_U, \Omega_\1}}^2}{\abs{\ev{\Omega_U, \Omega_\1}}} \leq \frac{\norm{\Omega_U - \Omega_\1}^2}{\abs{\ev{\Omega_U, \Omega_\1}}} < \frac{2\norm{\1 - U}^2}{2\gamma \norm{\pi(P)\Psi} - \delta^2},
\end{equation}
where we have used \eqref{eq:y_U-y_1} and \eqref{eq:|y_U,y_1|} in the last step.

Now let $\pi_P:P\fA P \rightarrow \B(\pi(P)\hilbH)$ be the compression of $\pi$ by $P$, which is a nonzero irreducible representation of $P\fA P$ by Lemma \ref{lem:compressed_rep}. Since $\ev{\Omega_U, \Omega_\1} \neq 0$ for every $U \in Z_{P,\gamma,\delta}$, Lemma \ref{cor:U_extension} yields a continuous map $\cU':Z_{P,\gamma,\delta} \rightarrow \Unitary(P\fA P)$ such that
\[
\pi_P(\cU'(U))\Omega_U = \Omega_\1
\]
and
\begin{align*}
\norm{P - \cU'(U)} &\leq \norm{\Omega_U - \Omega_\1} + \abs{1 - \frac{\ev{\Omega_U, \Omega_\1}}{\abs{\ev{\Omega_U, \Omega_\1}}}} \leq \beta\norm{\1 - U} < 2.
\end{align*}

We then define a continuous map $\cU'': Z_{P,\gamma, \delta} \rightarrow \Unitary(\fA)$ by
\[
\cU''(U) = \1 - P + \cU'(U).
\]
It is easy to check that this is indeed unitary. Furthermore,
\[
\norm{\1 - \cU''(U)} = \norm{P - \cU'(U)} < 2
\]
for all $U \in Z_{P,\gamma, \delta}$, so we can apply the principal branch of the logarithm to $\cU''(U)$ to obtain a continuous function $A:Z_{P,\gamma, \delta} \rightarrow \fA_\tn{sa}$, $A(U) = -i\Log \cU''(U)$ satisfying $e^{iA(U)} = \cU''(U)$ and $\norm{A(U)} < \pi$ for all $U \in Z_{P,\gamma, \delta}$.

We define $\cKW:Z_{P,\gamma, \delta} \times I \rightarrow \Unitary(\fA)$ by
\[
\cKW(U,s) = e^{isA(U)}U
\]
which is manifestly continuous. It is clear that \ref{ite:W(U,0)} holds. Since $\cU'(U) \in P\fA P$, we have $(\1 - P)\cU''(U) = (\1 - P)$, from which it follows that $(\1 - P)A(U) = 0$, hence $(\1 - P)e^{isA(U)} = (\1 - P)$ for all $U \in Z_{P,\gamma, \delta}$ and $s \in I$. Thus, \ref{ite:(1-P)W(U,s)} holds. 

From $(\1 - P)A(U) = 0$ it follows that $[P, A(U)] = 0$, hence $[P,e^{isA(U)}] = 0$ for all $U \in Z_{P,\gamma, \delta}$ and $s \in I$. Thus,
\begin{align*}
\norm{\pi(P\cKW(U,s))\Psi} = \norm{\pi(e^{isA(U)}PU)\Psi} = \norm{\pi(PU)\Psi},
\end{align*}
proving \ref{ite:PW(U,s)x}. Since $P\cU''(U) = \cU'(U) = \cU'(U)P$, we also have
\[
\pi(P\cKW(U,1))\Psi = \pi(\cKU'(U)U)\Psi = \norm{\pi(PU)\Psi} \pi_P(\cKU'(U))\Omega_U = \norm{\pi(PU)\Psi}\Omega_\1,
\]
for all $U \in Z_{P,\gamma, \delta}$, proving \ref{ite:PW(U,1)x}. Finally,
\begin{align*}
\norm{\1 - \cKW(U,s)} &\leq   \norm{\1 - U} + \norm{\1 - e^{isA(U)}}\\
&\leq  \norm{\1 - U} + \norm{\1 - \cU''(U)} \\
&\leq (1 + \beta)\norm{\1 - U},
\end{align*}
proving \ref{ite:W(u,s)_norm}.
\end{proof}

\section{Proof of the Equi-\texorpdfstring{$LC^n$}{LCn} Condition}
\label{sec:equiLCn}

Let $\fA$ be a unital $C^*$-algebra, let $P \in \fA$ be a projection, let $\psi \in \sP(\fA)$, and let $t \in [0,1]$. Define
\[
S(\psi, t) = \qty{U \in \Unitary(\fA): \psi(U^*PU) \geq t}.
\]
As preparation for using Theorem \ref{thm:finite-dim_Michael}, our goal in this section is to show that the collection $\cS$ of subsets $S(\psi, t)$ is equi-$LC^n$ for all $n$, when $\psi$ and $t$ range over suitable subsets of $\sP(\fA)$ and $I = [0,1]$, respectively. Using Proposition \ref{prop:equiLCn_topological_group}, it will suffice to show that $\cS$ is equi-$LC^n$ at $\1 \in \Unitary(\fA)$. Thus, for every $\varepsilon > 0$, we will need to find a $\delta > 0$ -- independent of $n$, $\psi$, and $t$ -- such that every continuous image of $\bbS^n$ in $B_\delta(\1) \cap S(\psi, t)$ is contractible in $B_\varepsilon(\1) \cap S(\psi, t)$. 

It will be convenient to work in the GNS representation $(\cH, \pi, \Psi)$ of $\psi$. Note that $S(\psi, t) = \{U \in \Unitary(\fA): \norm{\pi(PU)\Psi}^2 \geq t\}$. Given $\delta > 0$, we deform a family of unitaries in $B_\delta(\1) \cap S(\psi, t)$ to a point in two steps, while keeping the image of the deformation in $S(\psi, t)$ and in a small ball around $\1$. First, we deform the unitaries so that they always map $\Psi$ to a linear combination of two fixed unit vectors $\Omega \in \pi(P)\cH$ and $\Phi \in \pi(\1 - P)\cH$, where the coefficients of the linear combination are strictly positive. Lemma \ref{lem:steps_3_4_homotopy} then shows how to deform the resulting family of unitaries to a point using the results of Section \ref{sec:continuous_unitaries}.

For the first step described above, a different approach needs to be taken depending on whether $\norm{\pi(P)\Psi}$ is ``large'' or ``small.'' The case when $\norm{\pi(P)\Psi}$ is large is more straightforward and can be handled with Lemma \ref{lem:ZPgammadelta}; this is done in  Lemma \ref{lem:normsq_pi(P)x>2delta}. Lemma \ref{lem:normsq_pi(P)x<2delta} shows that the case when $\norm{\pi(P)\Psi}$ is small can be reduced to the case where $\pi(P)\Psi = 0$, which is dealt with in Lemma \ref{lem:pi(P)x=0}. This case is more interesting and requires the assumption that $\pi(P)\cH$ is infinite-dimensional. The latter entails that the unit sphere $\bbS \pi(P)\cH$ is contractible, whence the homotopy lifting property of a fiber bundle $\Unitary(P\fA P) \rightarrow \bbS \pi(P)\cH$ will allow us to deform the family of unitaries as desired.

We begin with the following lemma which will take care of arguments common to both cases.

\begin{lem}\label{lem:steps_3_4_homotopy}
Let $\fA$ be a unital $C^*$-algebra, let $P \in \fA$ be a projection, and let $(\cH, \pi)$ be a nonzero irreducible representation of $\fA$. Furthermore, let $\Phi,\Psi, \Omega \in \bbS \cH$ such that $\ev{\Psi,\Omega} \geq 0$, $\ev{\Phi, \Psi} > 0$, and $\ev{\Phi,\Omega} = 0$. Let $t \in [0,1)$. Define
\begin{equation}\label{eq:R_def}
R = \qty{U \in \Unitary(\fA): \ev{\Omega, \pi(U)\Psi} \in [\sqrt{t},1) \tn{ and } \ev{\Phi, \pi(U)\Psi} = \sqrt{1 - \ev{\Omega,\pi(U)\Psi}^2} }.
\end{equation}
Given $\zeta \in (0,1]$, either $B_\zeta(\1) \cap R = \varnothing$ or the inclusion map $B_\zeta(\1) \cap R \rightarrow B_{3\zeta}(\1) \cap R$ is nulhomotopic.
\end{lem}

\begin{proof}
Assume $B_\zeta(\1) \cap R \neq \varnothing$. Observe that for any $U \in R$, the inner products $\ev{\Omega,\pi(U)\Psi}$ and $\ev{\Phi,\pi(U)\Psi}$ square to one, and since $U \in \Unitary(\fA)$ implies $\norm{\pi(U)\Psi} = 1$, we have 
\begin{equation}\label{eq:R_expanded}
\pi(U)\Psi = \ev{\Omega,\pi(U)\Psi}\Omega + \ev{\Phi,\pi(U)\Psi}\Phi.
\end{equation}
Also note that $\ev{\Phi, \pi(U)\Psi}$ is strictly positive.

Let $Y_+$ and $\cU:Y_+ \rightarrow \Unitary(\fA)$ be as in Theorem \ref{thm:unitary_Kadison_lemma} and let $Z$ and $\cKV:Z \times I \rightarrow \Unitary(\fA)$ be as in Lemma \ref{lem:V_deform_to_U}. Since $U \in R$ implies
\[
\ev{\Psi, \pi(U)\Psi} = \ev{\Omega, \pi(U)\Psi}\ev{\Psi,\Omega} +  \ev{\Phi,\pi(U)\Psi}\ev{\Phi,\Psi} > 0,
\]
we see that $B_\zeta(\1) \cap R \subset Z$. Observe further that for all $(U, s) \in (B_\zeta(\1) \cap R) \times I$, part \ref{ite:V(U,s)x} of Lemma \ref{lem:V_deform_to_U} implies $\cKV(U,s) \in R$ and part \ref{ite:V(U,s)_norm} of Lemma \ref{lem:V_deform_to_U} implies $\cKV(U,s) \in B_{3\zeta}(\1)$. Thus, restricting the domain and range of $\cKV$ gives a homotopy $\cKV:(B_\zeta(\1) \cap R) \times I \rightarrow B_{3\zeta}(\1) \cap R$. By parts \ref{ite:V(U,0)} and \ref{ite:V(U,1)} of Lemma \ref{lem:V_deform_to_U}, we have $\cKV(U,0) = U$ and $\cKV(U,1) = \cU(\Psi, \pi(U)\Psi)$ for all $U \in B_\zeta(\1) \cap R$. Thus, the inclusion is homotopic to $U \mapsto \cU(\Psi, \pi(U)\Psi)$.

Now fix an arbitrary $U_0 \in B_\zeta(\1) \cap R$. Given $(U,s) \in (B_\zeta(\1) \cap R) \times I$, define
\[
\Upsilon_{U,s} = s\pi(U_0)\Psi + (1 - s)\pi(U)\Psi.
\]
By the triangle inequality we know $\norm{\Upsilon_{U,s}} \leq 1$. Since $\ev{\Phi,\Omega} = 0$, it is straightforward to see from \eqref{eq:R_expanded} that $\ev{\pi(U_0)\Psi, \pi(U)\Psi}$ is real and strictly positive for all $U \in R$. Thus,
\begin{align*}
\norm{\Upsilon_{U,s}} &= \sqrt{s^2 + (1 - s)^2 + 2s(1 - s)\Re \ev{\pi(U_0)\Psi, \pi(U)\Psi}} \\
&\geq \sqrt{s^2 + (1 - s)^2} \geq \frac{1}{\sqrt{2}}.
\end{align*}
Given $(U,s) \in (B_\zeta(\1) \cap R) \times I$, define 
\[
\widehat\Upsilon_{U,s} = \frac{\Upsilon_{U,s}}{\norm{\Upsilon_{U,s}}}
\]
and note that $(\Psi, \widehat\Upsilon_{U,s}) \in Y_+$.

Define $H:(B_\zeta(\1) \cap R) \times I \rightarrow B_{3\zeta}(\1) \cap R$ by
\begin{align*}
H(U,s) = \cU(\Psi,  \widehat\Upsilon_{U,s}).
\end{align*}
Let us make sure that $\cU(\Psi, \widehat\Upsilon_{U,s}) \in B_{3\zeta}(\1) \cap R$. We know $\pi(\cU(\Psi,\widehat\Upsilon_{U,s}))\Psi = \widehat\Upsilon_{U,s}$ by Theorem \ref{thm:unitary_Kadison_lemma}, hence 
\[
\ev{\Omega, \pi(\cU(\Psi, \widehat\Upsilon_{U,s}))\Psi} = \frac{s\ev{\Omega, \pi(U_0)\Psi} + (1 - s)\ev{\Omega,\pi(U)\Psi}}{\norm{\Upsilon_{U,s}}} \geq \sqrt{t}.
\]
We may similarly expand $\ev{\Phi, \pi(\cU(\Psi,\widehat\Upsilon_{U,s}))\Psi}$ and observe $\ev{\Phi, \pi(\cU(\Psi, \widehat\Upsilon_{U,s}))\Psi} > 0$. Since $\widehat\Upsilon_{U,s}$ is a unit vector in the span of $\Phi$ and $\Omega$, it follows that $\cU(\Psi, \widehat\Upsilon_{U,s}) \in R$ for all $(U,s) \in (B_\zeta(\1) \cap R) \times I$. Furthermore, by Theorem \ref{thm:unitary_Kadison_lemma} and Lemma \ref{lem:normalization_bound} we know
\begin{align*}
\norm{\1 - \cU(\Psi,\widehat\Upsilon_{U,s})} &= \norm{\Psi - \widehat\Upsilon_{U,s}} \leq \frac{\norm{\Psi - \Upsilon_{U,s}}}{\sqrt{\norm{\Upsilon_{U,s}}}}\\
&\leq 2^{1/4} \qty(s \norm{\Psi - \pi(U_0)\Psi} + (1 - s)\norm{\Psi - \pi(U)\Psi})\\
&< 2^{1/4}\zeta < 3\zeta,
\end{align*}
so indeed $\cU(\Psi, \widehat\Upsilon_{U,s}) \in B_{3\zeta}(\1)$ for all $(U,s) \in (B_\zeta(\1) \cap R) \times I$. Thus, $H$ is well defined, and it is  manifestly continuous. At $s = 0$, we have $H(U,0) = \cU(\Psi, \pi(U)\Psi) = \cKV(U,1)$ and at $s = 1$, we have $H(U,1) = \cU(\Psi, \pi(U_0)\Psi)$, which is independent of $U$. Thus, the map $U \mapsto \cU(\Psi, \pi(U)\Psi)$ is nulhomotopic, hence the inclusion is nulhomotopic.
\end{proof}

We proceed with the case when $\norm{\pi(P)\Psi}$ is ``large.''

\begin{lem}\label{lem:normsq_pi(P)x>2delta}
Let $\fA$ be a unital $C^*$-algebra, let $P \in \fA$ be a projection, and let $(\cH, \pi)$ be a nonzero irreducible representation of $\fA$. Furthermore, let $\Psi \in \bbS \cH$, let $t \in [0,1/4]$, and define
\[
S = \qty{U \in \Unitary(\fA): \norm{\pi(PU)\Psi}^2 \geq t}.
\]
Let $\delta \in (0,1/1296)$ and assume $\norm{\pi(P)\Psi}^2 \geq 2\delta$. Either $B_\delta(\1) \cap S = \varnothing$ or the inclusion $B_\delta(\1) \cap S \rightarrow B_{39\sqrt{\delta}}(\1) \cap S$ is nulhomotopic.
\end{lem}

\begin{proof}
Assume $B_\delta(\1) \cap S \neq \varnothing$. We show the inclusion $B_\delta(\1) \cap S \rightarrow B_{39\sqrt{\delta}}(\1) \cap S$ is nulhomotopic.

 If $\norm{\pi(P)\Psi}^2 \geq t + 2\delta$, then for every $U \in B_\delta(\1)$ we have
\[
\norm{\pi(PU)\Psi}^2 \geq \norm{\pi(P)\Psi}^2 - 2\delta \geq t
\]
where the first inequality follows by writing the squared norm as an inner product and then approximating the two $U$'s by identities. Thus, we see that $U \in S$, hence $B_\delta(\1) \cap S = B_\delta(\1)$. The space $B_\delta(\1)$ is contractible since $\delta \leq 2$, hence the inclusion $B_\delta(\1) \cap S \rightarrow B_{39\sqrt{\delta}}(\1) \cap S$ is nulhomotopic.

That leaves the case $\norm{\pi(P)\Psi}^2 < t + 2\delta$. In this case, we have
\begin{align*}
\norm{\pi(\1 - P)\Psi}^2 = 1 - \norm{\pi(P)\Psi}^2 > 1 - t - 2\delta > \frac{16}{25},
\end{align*}
where the last step follows since $t < 1/4$ and $\delta < 11/200$.
For any $U \in B_\delta(\1)$, we have
\[
\norm{\pi(\1 - P)\pi(U)\Psi - \pi(\1 - P)\Psi} \leq \norm{U - \1} < \delta,
\]
so by the triangle inequality
\[
\norm{\pi(\1 - P)\pi(U)\Psi} > \frac{4}{5} - \delta > \frac{45}{64},
\]
where the last step follows since $\delta < 31/320$.
Thus, we see that $U \in B_\delta(\1)$ implies $U \in Z_{\1 - P, \gamma_{\1-P}, \delta}$, with $\gamma_{\1 - P} = 45/64$ (note that the roles of $P$ and $\1 - P$ have switched relative to the notation of Lemma \ref{lem:ZPgammadelta}). Using $\norm{\pi(\1 - P)\Psi} > 4/5$ and $\delta^2 < 81/200$, we get that $\beta_{\1 - P} < 3$, where $\beta_{\1 - P}$ is defined by \eqref{eq:beta_def}. Using $\delta^2 < 9/8$ and $\delta < 2/3$ we verify \eqref{eq:beta_gamma_delta_condition}.

Let $\cKW_{\1-P} :Z_{\1-P,\gamma_{\1 - P},\delta} \times I \rightarrow \Unitary(\fA)$ be the map from Lemma \ref{lem:ZPgammadelta}. Our first homotopy of the inclusion is the map $H_1:(B_\delta(\1) \cap S) \times I\rightarrow B_{\sqrt{\delta}}(\1) \cap S$ defined by
\[
H_1(U,s) = \cKW_{\1 - P}(U,s).
\]
By \ref{ite:(1-P)W(U,s)} of Lemma \ref{lem:ZPgammadelta} we know $PH_1(U,s) = PU$, so indeed $H_1(U,s) \in S$ for all $(U,s) \in (B_\delta(\1) \cap S) \times I$. By \ref{ite:W(u,s)_norm} of Lemma \ref{lem:ZPgammadelta} and the fact that $\beta_{\1 - P} < 3$ we know
\begin{equation}\label{eq:H_1_norm}
\norm{\1 - H_1(U,s)} \leq 4\norm{\1 - U} < 4\delta,
\end{equation}
so in particular we indeed have $H_1(U,s) \in B_{39\sqrt{\delta}}(\1)$ for all $(U,s) \in (B_\delta(\1) \cap S) \times I$. By \ref{ite:W(U,0)} of Lemma \ref{lem:ZPgammadelta} we know $H_1(U,0) = U$ for all $U \in B_\delta(\1) \cap S$. Thus, $H_1$ is a well-defined homotopy of the inclusion map. At the end of the homotopy, \ref{ite:PW(U,1)x} of Lemma \ref{lem:ZPgammadelta} yields
\[
\pi((\1 - P)H_1(U,1))\Psi = \norm{\pi((\1 - P)U)\Psi} \cdot \frac{\pi(\1 - P)\Psi}{\norm{\pi(\1 - P)\Psi}}
\]
for all $U \in B_\delta(\1) \cap S$.

For the next part of the homotopy, we follow a similar procedure with $P$ instead of $\1-P$. By the triangle inequality and the assumption $\norm{\pi(P)\Psi}^2 \geq 2\delta$, we have
\begin{equation}\label{eq:gamma_P}
\norm{\pi(PU)\Psi} \geq \norm{\pi(P)\Psi} - \delta \geq \sqrt{2\delta} - \delta.
\end{equation}
for every $U \in B_\delta(\1)$. Define $\gamma_P = \sqrt{2\delta} - \delta$ and note that $\gamma_P > 0$. For $U \in B_\delta(\1) \cap S$, \eqref{eq:H_1_norm} and \eqref{eq:gamma_P} imply $H_1(U,1) \in Z_{P, \gamma_P, 4\delta}$ since $PH_1(U,1) = PU$. Observe that $\delta < 1/2$ implies
\[
\gamma_P \norm{\pi(P)\Psi} \geq 2\delta - \delta \sqrt{2\delta} > \delta,
\]
hence
\[
\beta_P < \frac{1}{\sqrt{\delta}} + \sqrt{\frac{2}{2\delta - \delta^2}} < \frac{3}{\sqrt{\delta}},
\]
where $\beta_P$ is defined by \eqref{eq:beta_def} and we have used $\delta^2 < 3\delta/2$ in the second step. Using $\delta < 2$ and $\sqrt{\delta} < 2/3$ we verify \eqref{eq:beta_gamma_delta_condition}. 

Let $\cKW_P:Z_{P, \gamma_P, 4\delta} \times I \rightarrow \Unitary(\fA)$ be the map from Lemma \ref{lem:ZPgammadelta}. We define our second homotopy as $H_2:(B_\delta(\1) \cap S) \times I \rightarrow B_{39\sqrt{\delta}}(\1) \cap S$ by
\begin{align*}
H_2(U,s) = \cKW_P(H_1(U,1), s).
\end{align*}
By \ref{ite:PW(U,s)x} of Lemma \ref{lem:ZPgammadelta}, we have
\[
\norm{\pi(PH_2(U,s))\Psi} = \norm{\pi(PH_1(U,1))\Psi} = \norm{\pi(PU)\Psi},
\]
so indeed $H_2(U,s) \in S$ for all $(U,s) \in (B_\delta(\1) \cap S) \times I$. By \ref{ite:W(u,s)_norm} of Lemma \ref{lem:ZPgammadelta} we have for all $(U,s) \in (B_\delta(\1) \cap S) \times I$:
\begin{align*}
\norm{\1 - H_2(U,s)} &\leq \qty(1 + \frac{3}{\sqrt{\delta}})\norm{\1 - H_1(U,1)} \leq \qty(4 + \frac{12}{\sqrt{\delta}})\norm{\1 - U}  \\
&< 4\delta + 12 \sqrt{\delta} < 13\sqrt{\delta}
\end{align*}
where we have used $\sqrt{\delta} < 1/4$ in the last step. Thus, $H_2(U,s) \in B_{39\sqrt{\delta}}(\1)$. By \ref{ite:W(U,0)} of Lemma \ref{lem:ZPgammadelta}, we know $H_2(U,0)= H_1(U,1)$. This shows that $H_2$ is a well-defined homotopy of $H_1(\,\_\,,1)$. At the end of the homotopy, \ref{ite:(1-P)W(U,s)} and \ref{ite:PW(U,1)x} of Lemma \ref{lem:ZPgammadelta} yield
\begin{align}
\pi(H_2(U,1))\Psi &=  \pi(P\cKW_P(H_1(U,1),1))\Psi + \pi((\1 - P)\cKW_P(H_1(U,1), 1))\Psi\\
&= \norm{\pi(PU)\Psi} \cdot \frac{\pi(P)\Psi}{\norm{\pi(P)\Psi}} + \norm{\pi((\1 - P)U)\Psi} \cdot \frac{\pi(\1 - P)\Psi}{\norm{\pi(\1 - P)\Psi}}. \label{eq:H_2(U,1)x}
\end{align}

Finally, we observe that $H_2(\,\_ \,, 1) : B_\delta(\1) \cap S \rightarrow B_{39\sqrt{\delta}}(\1) \cap S$ has image in $B_{13\sqrt{\delta}}(\1) \cap R$, where $R$ is as defined in Lemma \ref{lem:steps_3_4_homotopy} with $\Phi,\Omega \in \bbS \cH$ defined as $\Omega = \pi(P)\Psi/\norm{\pi(P)\Psi}$ and $\Phi = \pi(\1 - P)\Psi/\norm{\pi(\1 - P)\Psi}$. Noting that $R \subset S$, we see that $H_2(\,\_\, , 1)$ factors as a composition
\[
B_{\delta}(\1) \cap S \rightarrow B_{13\sqrt{\delta}}(\1) \cap R \rightarrow B_{39\sqrt{\delta}(\1)} \cap R \rightarrow B_{39\sqrt{\delta}}(\1) \cap S,
\]
where the second and third arrow are inclusion maps. By Lemma \ref{lem:steps_3_4_homotopy} we know the middle arrow is nulhomotopic, so $H_2(\, \_\, ,1):B_\delta(\1) \cap S \rightarrow B_{39\sqrt{\delta}}(\1) \cap S$ is nulhomotopic, as desired.
\end{proof}

Next, we handle the case where $\pi(P)\Psi = 0$.

\begin{lem}\label{lem:pi(P)x=0}
Let $\fA$ be a unital $C^*$-algebra, let $P \in \fA$ be a projection, and let $(\cH, \pi)$ be a nonzero irreducible representation of $\fA$ such that the range $\pi(P)\cH$ is infinite-dimensional. Furthermore, let $\Psi \in \bbS \cH$ satisfy $\pi(P)\Psi = 0$, let $t \in [0,1)$, and define
\[
S = \qty{U \in \Unitary(\fA): \norm{\pi(PU)\Psi}^2 \geq t}.
\]
Given $\delta \in (0, 7/16)$, either $B_\delta(\1) \cap S = \varnothing$ or the inclusion  $B_\delta(\1) \cap S \rightarrow B_{12\delta}(\1) \cap S$ is nulhomotopic.
\end{lem}

\begin{proof}
Let $\delta \in (0,7/16)$ and assume $B_\delta(\1) \cap S \neq \varnothing$. We show that the inclusion map $B_\delta(\1) \cap S \rightarrow B_{12\delta}(\1) \cap S$ is nulhomotopic. 

Suppose $t = 0$. Then $S = \Unitary(\fA)$, so $B_\delta(\1) \cap S = B_\delta(\1)$. Thus, the inclusion map $B_\delta(\1) \cap S \rightarrow B_{12\delta}(\1) \cap S$ is nulhomotopic because the domain $B_\delta(\1)$ is contractible.

Consider the case where $t > 0$. Let $p:\hilbH \rightarrow \pi(P)\hilbH$ be the projection onto $\pi(P)\hilbH$ and let $\pi_P:P\fA P \rightarrow \B(\pi(P)\hilbH)$ be defined by $\pi_P(A) = p\pi(A)p^*$. By Lemma \ref{lem:compressed_rep} we know $(\pi(P)\cH,  \pi_P)$ is a nonzero irreducible representation of $P\fA P$.

We have a continuous map
\[
\Omega:B_\delta(\1) \cap S \rightarrow \bbS \pi(P)\hilbH, \quad \Omega_U = \frac{\pi(PU)\Psi}{\norm{\pi(PU)\Psi}}. 
\]
Fix $U_0 \in B_\delta(\1) \cap S$ and set $\Omega_0 = \Omega_{U_0}$.
Since $\pi(P)\hilbH$ is infinite-dimensional, we know $\bbS\pi(P)\hilbH$ is contractible, so there exists a homotopy $\tilde E: (B_\delta(\1) \cap S) \times I \rightarrow \bbS \pi(P)\hilbH$ such that 
\[
\tilde E(U, 0) = \Omega_0 \qqtext{and} \tilde E(U, 1) = \Omega_U
\]
for all $U \in B_\delta(\1) \cap S$. Since $(\pi(P)\cH,  \pi_P)$ is a nonzero irreducible representation of $P\fA P$, the map
\[
\Unitary(P\fA P) \rightarrow \bbS \pi(P)\cH, \quad V \mapsto  \pi_P(V)\Omega_0
\]
is a fiber bundle by Corollary \ref{cor:principal_bundle}. Since the base space  $\bbS \pi(P)\cH$ is metrizable, it is in particular paracompact Hausdorff, so this fiber bundle is a Hurewicz fibration \cite{hurewicz1955HLP}. Therefore we may find a homotopy $E:(B_\delta(\1) \cap S) \times I \rightarrow \Unitary(P\fA P)$ such that the diagram below commutes.
\[
\begin{tikzcd}[column sep = large, row sep = large]
(B_\delta(\1) \cap S) \times \qty{0} \arrow[r, "\quad{(U, 0) \mapsto P}\quad"]\arrow[d, hook] & \Unitary(P\fA P)\arrow[d, "\substack{V\\\downmapsto\\ \pi_P(V)\Omega_0}"]\\
(B_\delta(\1) \cap S) \times I \arrow[r,"\tilde E"]\arrow[ru,dashed,"E"] & \bbS \pi(P)\cH
\end{tikzcd}
\]
In particular, $E(U, 0) = P$ and $\pi_P(E(U, 1))\Omega_0 = \Omega_U$ for all $U \in B_\delta(\1) \cap S$. By definition of $\pi_P$ and the fact that $E(U,1) = PE(U,1)$, we have 
\[
\Omega_U = \pi(P) \pi(E(U,1))\Omega_0 = \pi(E(U,1))\Omega_0 = \pi(E(U,1) + (\1 - P))\Omega_0.
\]
for all $U \in B_\delta(\1) \cap S$. 

We note that $E(U, s) + (\1 - P) \in \Unitary(\fA)$ for all $(U, s) \in (B_\delta(\1) \cap S) \times I$.  Define $H_1:(B_\delta(\1) \cap S) \times I \rightarrow B_{12\delta}(\1) \cap S$ by
\[
H_1(U, s) = \qty[E(U,s) + (\1 - P)]^*U\qty[E(U,s) + (\1 - P)].
\]
Since $E(U,s) + \1 - P$ is unitary, we see that 
\[
\norm{\1 - H_1(U,s)} = \norm{\1 - U} < \delta.
\]
Furthermore, since $\pi(E(U,s) + \1 - P)\Psi = \Psi$, we have
\begin{align*}
\ev{\Psi, \pi(H_1(U,s)^*PH_1(U,s))\Psi} &= \ev{\Psi, \pi(U^*E(U,s)PE(U,s)^* U)\Psi}\\
&= \ev{\Psi, \pi(U^*PU)\Psi} \geq t.
\end{align*}
so $H_1$ in fact has image in $B_\delta(\1) \cap S$. Since $E(U,0) = P$, we see that $H_1(U,0) = U$, so $H_1$ is a homotopy of the inclusion map. Finally, since $[P, E(U,s) + \1 - P] = 0$, we have
\[
\pi(PH_1(U,1))\Psi = \pi(E(U,1) + (\1 - P))^* \pi(PU)\Psi = \norm{\pi(PU)\Psi}\Omega_0
\]
for all $U \in B_\delta(\1) \cap S$.

We now proceed in a similar fashion to Lemma \ref{lem:normsq_pi(P)x>2delta}. By the triangle inequality we have
\begin{align*}
\norm{\pi((\1 - P)U)\Psi} > 1 - \delta > \frac{9}{16}
\end{align*}
for all $U \in B_\delta(\1)$, where we have used $\delta < 7/16$. Setting $\gamma = 9/16$, we see that $B_\delta(\1) \subset Z_{\1 - P, \gamma, \delta}$, where $Z_{\1 - P, \gamma, \delta}$ is as in Lemma \ref{lem:ZPgammadelta}. Using $\norm{\pi(\1 - P)\Psi} = 1$ and $\delta^2 < 81/200$, we see that $\beta < 3$, where $\beta$ is defined by \eqref{eq:beta_def}. Using $\delta^2 < 9/8$ and $\delta < 2/3$, we verify \eqref{eq:beta_gamma_delta_condition}.

Let $\cKW:Z_{\1 - P, \gamma, \delta} \times I \rightarrow \Unitary(\fA)$ be the map from Lemma \ref{lem:ZPgammadelta}. Our second homotopy is the map $H_2:(B_\delta(\1) \cap S) \times I \rightarrow B_{12\delta}(\1) \cap S$ defined by
\[
H_2(U,s) = \cKW(H_1(U,1), s).
\]
By \ref{ite:(1-P)W(U,s)} of Lemma \ref{lem:ZPgammadelta}, we have
\[
PH_2(U,s) = H_1(U,1),
\]
hence 
\[
\norm{\pi(PH_2(U,s))\Psi} = \norm{\pi(PH_1(U,1))\Psi}= \norm{\pi(PU)\Psi},
\]
so indeed $H_2(U,s) \in S$ for all $(U,s) \in (B_\delta(\1) \cap S) \times I$. By \ref{ite:W(u,s)_norm} of Lemma \ref{lem:ZPgammadelta} and the fact that $\beta < 3$, we know
\[
\norm{\1 - H_2(U,s)} \leq 4 \norm{\1 - H_1(U,1)} < 4\delta,
\] 
hence $H_2(U,s) \in B_{12\delta}(\1)$ for all $(U,s) \in (B_\delta(\1) \cap S) \times I$. This proves that $H_2$ is well-defined. At $s = 0$, \ref{ite:W(U,0)} of Lemma \ref{lem:ZPgammadelta} yields $H_2(U,0) = H_1(U,1)$. At $s = 1$, \ref{ite:(1-P)W(U,s)} and \ref{ite:PW(U,1)x} of Lemma \ref{lem:ZPgammadelta} yield
\begin{align*}
\pi(H_2(U,1))\Psi &= \pi(P\cKW(H_1(U,1),1))\Psi + \pi((\1 - P)\cKW(H_1(U,1),1))\Psi\\
&= \pi(PH_1(U,1))\Psi + \norm{\pi((\1 - P)H_1(U,1))\Psi}\Psi\\
&= \norm{\pi(PU)\Psi}\Omega_0 + \norm{\pi((\1 - P)U)\Psi}\Psi
\end{align*}

Finally, we observe that $H_2(\,\_\,,1):B_\delta(\1) \cap S \rightarrow B_{12\delta}(\1) \cap S$ has image in $B_{4\delta}(\1) \cap R$, where $R$ is as defined in Lemma \ref{lem:steps_3_4_homotopy} with $\Omega = \Omega_0$ and $\Phi = \Psi$. Noting that $R \subset S$, we see that $H_2(\,\_\,,1)$ factors as a composition
\[
B_\delta(\1) \cap S \rightarrow B_{4\delta}(\1) \cap R \rightarrow B_{12\delta}(\1) \cap R \rightarrow B_{12\delta}(\1) \cap S,
\]
where the second and third arrow are inclusion maps. By Lemma \ref{lem:steps_3_4_homotopy} we know the middle arrow is nulhomotopic, so $H_2(\,\_\,,1)$ is nulhomotopic, as desired.
\end{proof}

The following lemma deals with the case where $\norm{\pi(P)\Psi}$ is ``small'' by reducing it to the case where $\pi(P)\Psi = 0$.

\begin{lem}\label{lem:normsq_pi(P)x<2delta}
Let $\fA$ be a unital $C^*$-algebra, let $P \in \fA$ be a projection, and let $(\cH, \pi)$ be a nonzero irreducible representation of $\fA$ such that the range of $\pi(P)\cH$ is infinite-dimensional. Furthermore, let $\Psi \in \bbS \cH$, let $t \in [0,1)$, and define
\[
S = \qty{U \in \Unitary(\fA) : \norm{\pi(PU)\Psi}^2 \geq t}.
\]
Suppose $\delta \in (0, 1/36)$ and $\norm{\pi(P)\Psi}^2 < 2\delta$. Either $B_\delta(\1) \cap S = \varnothing$ or the inclusion map $B_\delta(\1) \cap S \rightarrow B_{32\sqrt{\delta}}(\1) \cap S$ is nulhomotopic.
\end{lem}

\begin{proof}
We know
\[
\norm{\pi(\1 - P)\Psi}^2 = 1 - \norm{\pi(P)\Psi}^2 > 1 - 2\delta > 0,
\]
so $\pi(\1 - P)\Psi \neq 0$.
By Corollary \ref{cor:unitary_onto_projection}, there exists $V \in \Unitary(\fA)$ such that
\[
\pi(V)\Psi = \frac{\pi(\1 - P)\Psi}{\norm{\pi(\1 - P)\Psi}}
\]
and
\[
\norm{\1 - V} = \norm{\frac{\pi(\1 - P)\Psi}{\norm{\pi(\1 - P)\Psi}} - \Psi} \leq \frac{\norm{\pi(\1 - P)\Psi - \Psi}}{\sqrt{\norm{\pi(\1 - P)\Psi}}} < \sqrt{\frac{2\delta}{1 - 2\delta}},
\]
where we have used Lemma \ref{lem:normalization_bound}. Let $\gamma = \sqrt{2\delta/(1 - 2\delta)}$. Note that $\delta + \gamma < 7/16$.

Assume $B_\delta(\1) \cap S \neq \varnothing$. Define
\[
S' = \qty{U \in \Unitary(\fA): \norm{\pi(PU)\pi(V)\Psi}^2 \geq t}.
\]
Consider the map $f:B_\delta(\1) \cap S \rightarrow B_{\delta + \gamma}(\1) \cap S'$ defined by
\[
f(U) = UV^*,
\]
which is easily seen to be well-defined. Since $\pi(PV)\Psi = 0$ and $\delta + \gamma < 7/16$, Lemma \ref{lem:pi(P)x=0} implies that the inclusion map $B_{\delta + \gamma}(\1) \cap S' \rightarrow B_{12\delta + 12\gamma}(\1) \cap S'$ is nulhomotopic. Composing this inclusion with $f$ therefore yields a nulhomotopic map $g:B_\delta(\1) \cap S \rightarrow B_{12\delta + 12\gamma}(\1) \cap S'$, $g(U) = UV^*$. Finally, we compose $g$ with the map $h:B_{12\delta + 12\gamma}(\1) \cap S' \rightarrow B_{12\delta + 13\gamma}(\1) \cap S$ defined by
\[
h(W) = WV.
\]
Then $h \circ g$ is nulhomotopic and is the inclusion $B_{\delta}(\1) \cap S \rightarrow B_{12\delta + 13\gamma}(\1) \cap S$. Using $\delta < 1/4$ we get $12 \delta + 13\gamma < 32\sqrt{\delta}$, so the inclusion $B_\delta(\1) \cap S \rightarrow B_{32\sqrt{\delta}}(\1) \cap S$ is nulhomotopic.
\end{proof}

Combining the cases considered above yields the following corollary.

\begin{cor}\label{cor:cases_combined}
Let $\fA$ be a unital $C^*$-algebra, let $P \in \fA$ be a projection, and let $(\cH, \pi)$ be a nonzero irreducible representation of $\fA$ such that the range of $\pi(P)\cH$ is infinite-dimensional. Let $\Psi \in \bbS \cH$, let $t \in [0,1/4]$, and define
\[
S = \qty{U \in \Unitary(\fA): \norm{\pi(PU)\Psi}^2 \geq t}.
\]
Let $\delta \in (0, 1/1296)$. Either $B_\delta(\1) \cap S = \varnothing$ or the inclusion $B_\delta(\1) \cap S \rightarrow B_{39\sqrt{\delta}}(\1) \cap S$ is nulhomotopic.
\end{cor}

\begin{proof}
If $\norm{\pi(P)\Psi}^2 \geq 2\delta$, then we apply Lemma \ref{lem:normsq_pi(P)x>2delta}. Otherwise, we apply Lemma \ref{lem:normsq_pi(P)x<2delta}.
\end{proof}

Finally, we show that the family $\cS$ of sets $S(\psi, t)$ is equi-$LC^n$ for all $n$.

\begin{thm}\label{thm:equiLCn_pure_states}
Let $\fA$ be a unital $C^*$-algebra, let $P \in \fA$ be a projection, and let $\sQ \subset \sP(\fA)$ be a nonempty family of pure states of $\fA$ such that:
\begin{enumerate}
	\item the range $\pi_\psi(P)\cH_\psi$ is infinite-dimensional for all $\psi \in \sQ$, where $(\cH_\psi, \pi_\psi, \Psi_\psi)$ is the GNS representation of $\psi$,
	\item $\psi \in \sQ$ and $U \in \Unitary(\fA)$ implies $\psi \circ \Ad U^* \in \sQ$
\end{enumerate}
For each $\psi \in \sQ$ and $t \in [0,1/4]$, define
\[
S(\psi, t) = \qty{U \in \Unitary(\fA): \psi(U^*PU) \geq t}.
\]
Then each $S(\psi, t)$ is nonempty and closed, and the family 
\[
\cS = \qty{S(\psi, t): \psi \in \sQ, t \in [0,1/4]}
\]
covers $\Unitary(\fA)$ and is equi-$LC^n$ for all $n \in \bbN \cup \qty{0, -1}$.
\end{thm}

\begin{proof}
Let $\psi \in \sQ$ and $t \in [0,1/4]$. Since $\pi_\psi(P)\cH_\psi$ is nonzero, there exists a unit vector $\Omega \in \pi_\psi(P)\cH_\psi$. By the Kadison transitivity theorem, there exists $U \in \Unitary(\fA)$ such that $\pi_\psi(U)\Psi_\psi = \Omega$. Then $\psi(U^*PU) = 1$, hence $U \in S(\psi, t)$. This proves that $S(\psi, t)$ is nonempty. Since $S(\psi, 0) = \Unitary(\fA)$ for every $\psi \in \sQ$, it is trivial that $\cS$ covers $\Unitary(\fA)$. Since $U \mapsto \psi(U^*PU)$ is a continuous function for all $\psi \in \sQ$, we see that $S(\psi, t)$ is closed for all $\psi \in \sQ$ and $t \in [0,1/4]$.

It is easy to see that for every $\psi \in \sQ$, $t \in [0,1/4]$, and $V \in \Unitary(\fA)$, we have
\[
S(\psi, t)V^* = S(\psi \circ \Ad(V^*), t),
\]
which is in $\cS$ since $\psi \circ \Ad(V^*) \in \sQ$ by definition of $\sQ$. Thus, by Proposition \ref{prop:equiLCn_topological_group}, it suffices to prove that $\cS$ is equi-$LC^n$ at the identity $\1 \in \Unitary(\fA)$.

Let $\varepsilon > 0$. Choose a positive number $\delta < \min(1/1296, (\varepsilon/39)^2)$. We claim that, for every $\psi \in \sQ$ and $t \in [0,1/4]$, every continuous image of the $k$-sphere $\bbS^k$ in $B_\delta(\1) \cap S(\psi, t)$ is contractible in $B_\varepsilon(\1) \cap S(\psi, t)$. If $B_\delta(\1) \cap S(\psi, t) = \varnothing$, then this is trivially. Otherwise, we note that 
\[
S(\psi, t) = \qty{U\in \Unitary(\fA): \norm{\pi_\psi(PU)\Psi_\psi}^2 \geq t},
\]
therefore Corollary \ref{cor:cases_combined} implies that the inclusion map $B_\delta(\1) \cap S(\psi, t) \rightarrow B_{39\sqrt{\delta}}(\1) \cap S(\psi, t)$ is nulhomotopic. Since $39\sqrt{\delta} < \varepsilon$, the inclusion map $B_\delta(\1) \cap S(\psi, t) \rightarrow B_\varepsilon(\1) \cap S(\psi, t)$ is nulhomotopic. Thus, every continuous image of the $k$-sphere in $B_\delta(\1) \cap S(\psi, t)$ is contractible in $B_\varepsilon(\1) \cap S(\psi, t)$, as desired.
\end{proof}

\section{Applying the Finite-Dimensional Selection Theorem}
\label{sec:michael_application}

We will now use the equi-$LC^n$ property proven in Theorem \ref{thm:equiLCn_pure_states} and Michael's finite-dimensional selection theorem \ref{thm:finite-dim_Michael} to prove our main results. First, given a finite-dimensional compact Hausdorff space $X$, a projection $P \in \fA$, and a weak*-continuous family of pure states $\psi:X \rightarrow \sP(\fA)$, we show how $\psi$ can be deformed to a weak*-continuous family $\omega:X \rightarrow \sP(\fA)$ by acting with a norm-continuous homotopy of unitaries $U:X \times I \rightarrow \Unitary(\fA)$ such that at the end of the deformation we have $\omega_x(P) = 1$ for all $x \in X$. For this problem, the following definition is natural.

\begin{defn}
Let $\fA$ be a unital $C^*$-algebra, let $X$ be a topological space, and let $A \subset X$. A \textit{$\Unitary(\fA)$-homotopy relative to $A$} is a continuous map $U:X\times I \rightarrow \Unitary(\fA)$ such that
\begin{enumerate}
	\item $U(x,0) = \1$ for all $x \in X$,
	\item $U(x,s) = \1$ for all $x \in A$ and $s \in I$.
\end{enumerate}
We say two weak*-continuous maps $\psi:X \rightarrow \sP(\fA)$ and $\omega:X \rightarrow \sP(\fA)$ are \textit{$\Unitary(\fA)$-homotopic relative to $A$} if there exists a $\Unitary(\fA)$-homotopy relative to $A$ such that $\omega_x = U(x,1)\psi_x$ for all $x \in X$, where $U(x,1)\psi_x = \psi_x \circ \Ad(U(x,1)^*)$, as in \eqref{eq:algebra_action_state}. In this case we see that $(x,s) \mapsto U(x,s)\psi_x$ is a weak*-continuous homotopy from $\psi$ to $\omega$ relative to $A$ in the usual sense.
\end{defn}

\begin{thm}\label{thm:Michael_application}
Let $X$ be a compact Hausdorff space with finite Lebesgue covering dimension. Let $\fA$ be a unital $C^*$-algebra, let $P \in \fA$ be a projection, and let $\psi:X \rightarrow \sP(\fA)$ be a weak*-continuous function such that the range $\pi_{x}(P)\cH_{x}$ is infinite-dimensional for all $x \in X$, where $(\cH_x, \pi_x,\Psi_x)$ is the GNS representation of $\psi_x$. There exists a $\Unitary(\fA)$-homotopy $U:X \times I \rightarrow \Unitary(\fA)$ relative to $\qty{x \in X: \psi_x(P) = 1}$ such that 
\begin{equation}\label{eq:end_of_homotopy}
\psi_x(U(x,1)^*PU(x,1)) = 1
\end{equation}
for all $x \in X$.
\begin{comment}
There exists a homotopy $U:X \times I \rightarrow \Unitary(\fA)$ such that 
\begin{enumerate}
	\item $U(x,0) = \1$ for all $x \in X$,
	\item $\psi_x(U(x,1)^*PU(x,1)) = 1$ for all $x \in X$,
	\item $U(x,t) = \1$ for all $t \in I$ and $x \in X$ such that $\psi_x(P) = 1$.
\end{enumerate}
\end{comment}
\end{thm}

\begin{proof}
Let $\sQ = \qty{\psi_x \circ \Ad U^*: x \in X \tn{ and } U \in \Unitary(\fA)}$. Let $x \in X$, $U \in \Unitary(\fA)$, and let $(\cH, \pi, \Psi)$ be the GNS representation of $\psi_x \circ \Ad U^*$. By uniqueness of the GNS construction up to unitary equivalence there exists a unitary $V:\cH_x \rightarrow \cH$ such that $V\pi_x(U)\Psi_x = \Psi$ and $\Ad V \circ \pi_x = \pi$. In particular $\pi(P)V = V \pi_x(P)$, which implies that $\pi(P)\cH = V\pi_x(P)\cH_x$ is infinite-dimensional. If $W \in \Unitary(\fA)$, then $\psi_x \circ \Ad U^* \circ \Ad W^* = \psi_x \circ \Ad (WU)^* \in \sQ$ by definition of $\sQ$. We see that $\sQ$ satisfies the hypotheses of Theorem \ref{thm:equiLCn_pure_states}.

Given $\psi \in \sQ$ and $t \in [0,1/4]$, let $S(\psi, t)$ be as in Theorem \ref{thm:equiLCn_pure_states}. Similarly, let $\cS$ be as in Theorem \ref{thm:equiLCn_pure_states}. Consider the carrier $\phi:X \times [0,1/4] \rightarrow \cP(\Unitary(\fA)) \setminus \qty{\varnothing}$ defined by
\[
\phi(x, t) = S(\psi_x, t).
\]
Theorem \ref{thm:equiLCn_pure_states} tells us that $\cS$ is a family of nonempty closed subsets of $\Unitary(\fA)$ and $\cS$ is equi-$LC^n$ for all $n \in \bbN \cup \qty{0, -1}$. To apply Theorem \ref{thm:finite-dim_Michael} we must show that this carrier is lower semicontinuous.

Let $(x_0, t_0) \in X \times [0,1/4]$, let $U_0 \in \phi(x_0, t_0)$, and let $\varepsilon > 0$. Choose $\delta > 0$ small enough such that $\delta < 1/2$ and 
\[
\sqrt{2 - 2\sqrt{1 - 2\delta}} < \varepsilon.
\]
There exists a neighborhood $O$ of $(x_0, t_0) \in X \times [0,1/4]$ such that $(x, t) \in O$ implies
\[
\abs{t - t_0} < \delta \qqtext{and} \abs{\psi_{x}(U_0^*PU_0) - \psi_{x_0}(U_0^*PU_0)} < \delta.
\]

Let $(x,t) \in O$. We need to show that $B_\varepsilon(U_0) \cap \phi(x,t) \neq \varnothing$. If $\psi_x(U_0^*PU_0) \geq t$, then $U_0 \in B_\varepsilon(U_0) \cap \phi(x,t)$ and we're done, so suppose $\psi_x(U_0^*PU_0) < t$. 

If $\psi_x(U_0^*PU_0) > 0$, then set
\[
\Omega = \frac{\pi_x(PU_0)\Psi_x}{\norm{\pi_x(PU_0)\Psi_x}} = \frac{\pi_x(PU_0)\Psi_x}{\sqrt{\psi_x(U_0^*PU_0)}}.
\] 
If $\psi_x(U_0^*PU_0) = 0$, then choose any unit vector $\Omega \in \pi_x(P)\hilbH_x$. In both cases we have
\[
\ev{\Omega, \pi_x(U_0)\Psi_x} = \sqrt{\psi_x(U_0^*PU_0)}.
\]

For $\theta \in [0,\pi/2]$, define
\[
\Omega_\theta = \frac{\sin \theta \,\Omega + \cos \theta\,\pi_x(U_0)\Psi_x}{\norm{ \sin \theta \,\Omega + \cos \theta \,\pi_x(U_0)\Psi_x}}.
\]
Note that
\[
\norm{\sin \theta \,\Omega + \cos \theta\, \pi_x(U_0)\Psi_x} = \sqrt{1 + \ev{\Omega, \pi_x(U_0)\Psi_x}\sin 2 \theta} \geq 1
\]
so $\Omega_\theta$ is always a well-defined unit vector. 

Let's compute a few helpful quantities. First,
\[
\ev{\Omega_\theta, \pi_x(U_0)\Psi_x} = \frac{\sin \theta \ev{\Omega, \pi_x(U_0)\Psi_x} + \cos \theta}{\sqrt{1 + \ev{\Omega, \pi_x(U_0)\Psi_x} \sin 2\theta}}.
\]
It will be useful to have the square of this:
\begin{align*}
\ev{\Omega_\theta, \pi(U_0)\Psi_x}^2 &= \frac{\sin^2 \theta \ev{\Omega, \pi_x(U_0)\Psi_x}^2 + \sin 2 \theta \ev{\Omega, \pi_x(U_0)\Psi_x} + \cos^2 \theta}{1 + \ev{\Omega, \pi_x(U_0)\Psi_x}\sin 2 \theta}\\
&= 1 - \qty(\frac{\sin^2 \theta}{1 + \ev{\Omega, \pi_x(U_0)\Psi_x}\sin 2 \theta})\qty(1 - \psi_x(U_0^*PU_0))
\end{align*}
Next,
\begin{align*}
\ev{\Omega_\theta, \pi_x(P)\Omega_\theta} &= \frac{\sin^2 \theta +  \sin 2\theta \ev{\Omega, \pi_x(U_0)\Psi_x} + \cos^2 \theta \, \psi_x(U_0^*PU_0)}{1 + \ev{\Omega, \pi_x(U_0)\Psi_x}\sin 2\theta} \\
&= 1 - \qty(\frac{\cos^2 \theta}{1 + \ev{\Omega, \pi_x(U_0)\Psi_x} \sin 2 \theta})\qty(1 - \psi_x(U_0^*PU_0)).
\end{align*}

We see that at $\theta = 0$ we have $\ev{\Omega_\theta, \pi_x(P)\Omega_\theta} = \psi_x(U_0^*PU_0) < t$ and at $\theta = \pi/2$ we have $\ev{\Omega_\theta, \pi_x(P)\Omega_\theta} = 1$. By the intermediate value theorem there exists $\theta \in [0,\pi/2]$ such that 
\[
\ev{\Omega_\theta, \pi_x(P)\Omega_\theta} = t.
\] 
With a little rearranging, we get
\begin{align*}
1 - t &= \frac{1 - \sin^2 \theta}{1 + \ev{\Omega, \pi_x(U_0)\Omega}\sin 2\theta} (1 - \psi_x(U_0^*PU_0)) \\
&= \frac{1 - \psi_x(U_0^*PU_0)}{1 + \ev{\Omega, \pi_x(U_0)\Omega}\sin 2 \theta} - 1 + \ev{\Omega_\theta, \pi_x(U_0)\Psi_x}^2\\
&\leq -\psi_x(U_0^*PU_0) + \ev{\Omega_\theta, \pi_x(U_0)\Psi_x}^2.
\end{align*}
Observe that
\[
\psi_x(U_0^*PU_0) > \psi_{x_0}(U_0^*PU_0) - \delta \geq t_0 - \delta > t - 2\delta.
\]
Thus,
\[
1 - t < 2\delta - t + \ev{\Omega_\theta, \pi_x(U_0)\Psi_x}^2,
\]
so $0 < \sqrt{1 - 2\delta} < \ev{\Omega_\theta, \pi_x(U_0)\Psi_x}$. Finally,
\[
\norm{\Omega_\theta - \pi_x(U_0)\Psi_x} = \sqrt{2 - 2\ev{\Omega_\theta, \pi_x(U_0)\Psi_x}} < \sqrt{2 - 2\sqrt{1 - 2\delta}} < \varepsilon.
\]

By Theorem \ref{thm:unitary_Kadison_lemma}, there exists $V \in \Unitary(\fA)$ such that $\pi_x(V)\pi_x(U_0)\Psi_x = \Omega_\theta$ and $\norm{\1 - V} < \varepsilon$. Thus, $VU_0 \in B_\varepsilon(U_0)$ and
\[
\psi_x(U_0^*V^*PVU_0) = \ev{\Omega_\theta, \pi_x(P)\Omega_\theta} = t,
\]
so $VU_0 \in \phi(x,t)$. This proves lower semicontinuity of $\phi$.

We are nearly ready to apply the finite-dimensional selection theorem. Define 
\[
A = (\qty{x \in X: \psi_x(P) = 1} \times [0,1/4]) \cup (X \times \qty{0})
\]
and note that $A$ is closed in $X \times [0,1/4]$. Since $X$ is compact and finite-dimensional, we know $d \defeq \dim(X \times [0,1/4]) < \infty$ \cite[Thm.~3.2.13]{EngelkingDimensionTheory}, hence any subset $B \subset (X \times [0,1/4]) \setminus A$ that is closed in $X \times [0,1/4]$ has $\dim B \leq d$ \cite[Thm.~3.1.23]{EngelkingDimensionTheory}. We observe that the constant map $U:A \rightarrow \Unitary(\fA)$, $U(x, t) = \1$ is a selection for $\phi|_A$. By Theorem \ref{thm:finite-dim_Michael} there exists a neighborhood $O$ of $A$ and an extension of $U$ to a selection for $\phi|_{O}$. Let us call this extension $U:O \rightarrow \Unitary(\fA)$ as well.

Since $X$ is compact and $O$ contains $X \times \qty{0}$, the tube lemma states that there exists $\varepsilon \in (0,1/4]$ such that $X \times [0, \varepsilon] \subset O$. Let us now restrict $U$ to $X \times [0,\varepsilon]$ and disregard how it may have been defined outside this set. We will extend $U$ to $X \times I$ using Theorem \ref{thm:weak*_selection_theorem}.

Define $\omega:X \rightarrow \sP(\fA)$ by $\omega_x = \psi_x \circ \Ad U(x, \varepsilon)^*$. Then $\omega$ is weak*-continuous (see for example \cite[Prop.~4.5]{beaudry2023homotopical}) and $\omega_x(P) \geq \varepsilon$ since $U(x,\varepsilon) \in \phi(x, \varepsilon)$ for all $x \in X$. By Theorem \ref{thm:weak*_selection_theorem} there exists $V:X \rightarrow \Unitary(\fA)$ such that $\omega_x(V_x^*PV_x) = 1$ for all $x \in X$ and 
\[
\norm{\1 - V_x} = \sqrt{2 - 2\sqrt{\omega_x(P)}} < 2.
\]
Let $T_x = -i\Log V_x$ for all $x \in X$, where the principal branch of the logarithm is used. Define $U:X \times [\varepsilon, 1] \rightarrow \Unitary(\fA)$ by
\[
U(x, t) = e^{iT_x(t - \varepsilon)/(1 - \varepsilon)}U(x, \varepsilon)
\]
It is clear that our previous function $U$ on $X \times [0,\varepsilon]$ can be glued to the function $U$ above to give a continuous function $U:X \times [0,1] \rightarrow \Unitary(\fA)$.

By definition of $U$ on $A$, we see that $U(x, 0) = \1$ for all $x \in X$. Furthermore,
\[
\psi_x(U(x,1)^*PU(x,1)) = \psi_x(U(x,\varepsilon)^*V_x^*PV_xU(x,\varepsilon)) = \omega_x(V_x^*PV_x) = 1
\]
for all $x \in X$. Finally, suppose $x \in X$ and $\psi_x(P) = 1$. Then for any $t \in [0,\varepsilon]$ we have $(x,t) \in A$, hence $U(x,t) = \1$. In particular, $U(x, \varepsilon) = \1$, hence $\omega_x(P) = \psi_x(P) = 1$, $V_x = \1$, and $T_x = 0$. Then for any $t \in [\varepsilon, 1]$, we see that $U(x,t) = \1$ by definition of $U(x,t)$. This proves that $U$ has all the desired properties.
\end{proof}

The requirement that $\pi_x(P)\cH_x$ is infinite-dimensional is automatically satisfied if $\fA$ is infinite-dimensional, unital, and simple, as follows from \cite[Thm.~2.4.9]{MurphyCAOT}. We record this observation as a corollary.

\begin{cor}\label{cor:infinite_simple}
Let $X$ be a compact Hausdorff space with finite Lebesgue covering dimension. Let $\fA$ be an infinite-dimensional, unital, simple $C^*$-algebra, let $P \in \fA$ be a projection, and let $\psi:X \rightarrow \sP(\fA)$ be weak*-continuous. There exists a $\Unitary(\fA)$-homotopy $U:X \times I \rightarrow \Unitary(\fA)$ relative to $\qty{x \in X: \psi_x(P) = 1}$ such that 
\begin{equation}
\psi_x(U(x,1)^*PU(x,1)) = 1
\end{equation}
for all $x \in X$.
\begin{comment}
There exists a homotopy $U:X \times I \rightarrow \Unitary(\fA)$ such that
\begin{enumerate}
	\item $U(x,0) = \1$ for all $x \in X$,
	\item $\psi_x(U(x,1)^*PU(x,1)) = 1$ for all $x \in X$,
	\item $U(x,t) = \1$ for all $t \in I$ and $x \in X$ such that $\psi_x(P) = 1$.
\end{enumerate}
\end{comment}
\end{cor}

%\begin{thm}
%Let $\fA$ be an infinite-dimensional, unital, simple $C^*$-algebra. Let $(P_n)_{n \in \bbN}$ be a decreasing sequence of projections $\fA$ and suppose there exists a unique pure state $\omega \in \sP(\fA)$ such that $\omega(P_n) = 1$ for all $n \in \bbN$. Let $X$ be a compact  Hausdorff space with finite Lebesgue covering dimension, let $x_0 \in X$, and let  $\psi:X \rightarrow \sP(\fA)$ be a weak*-continuous function such that $\psi_{x_0} = \omega$. Then $\psi$ is nulhomotopic relative to the basepoint $x_0$.
%\end{thm}

\begin{comment}
We now apply Corollary \ref{cor:infinite_simple} to prove our main result. The following lemma will be useful.

\begin{lem}[{\cite[Lem.~3.4.2]{BrownOzawa}}]\label{lem:restricted_state_factorized}
Let $\fA$ be a $C^*$-algebra and let $\fB$ be a $C^*$-subalgebra of $\fA$. If $\omega \in \sS(\fA)$ and $\omega|_{\fB} \in \sP(\fB)$, then
\[
\omega(AB) = \omega(A)\omega(B)
\]
for all $A \in \fB'$ and $B \in \fB$, where $\fB'$ is the commutant of $\fB$.
\end{lem}
\end{comment}

We now show how this result can be iterated over to obtain a nulhomotopy of $\psi$. The iteration will be accomplished by using a sequence of projections that excise the state $\psi$.  Recall from \cite{AkemannAndersonPedersenExcision} that if $\fA$ is a $C^*$-algebra, then a net $(E_\lambda)_{\lambda \in \Lambda}$ of positive, norm one elements in $\fA$ is said to \textit{excise} a state $\omega :\fA \rightarrow \bbC$ if for every $A \in \fA$, one has
\[
\lim_{\lambda} \norm{E_\lambda A E_\lambda - \omega(A)E_\lambda^2} = 0. 
\]

Also recall that a unital $C^*$-algebra has \textit{real rank zero} if the set of invertible self-adjoint elements is dense in the set of all self-adjoint elements. A non-unital $C^*$-algebra is said to have real rank zero if its unitization has real rank zero. Finally, recall that a $C^*$-algebra is \textit{$\sigma$-unital} if it admits a countable approximate identity.

\begin{prop}\label{prop:excising_sequence}
Let $\fA$ be a unital $C^*$-algebra with real rank zero. If $\omega \in \sP(\fA)$, $\fN = \qty{A \in \fA: \omega(A^*A) = 0}$ is its associated left ideal, and the associated hereditary $C^*$-subalgebra $\fN \cap \fN^*$ is  $\sigma$-unital, then $\omega$ can be excised by a decreasing sequence of projections $(P_n)_{n \in \bbN}$ such that $\omega(P_n) = 1$ for all $n \in \bbN$.
\end{prop}

\begin{proof}
Since $\fN \cap \fN^*$ is a hereditary $C^*$-subalgebra of $\fA$, we know $\fN \cap \fN^*$ has real rank zero by \cite[Cor.~2.8]{BrownPedersenRealRankZero}. The result then follows by synthesis of \cite[Prop.~2.9]{BrownPedersenRealRankZero} and \cite[Prop.~2.2]{AkemannAndersonPedersenExcision}.
\end{proof}

\begin{thm}\label{thm:iteration_nulhomotopy}
Let $X$ be a compact Hausdorff space with finite Lebesgue covering dimension. Let $\fA$ be a unital $C^*$-algebra, and let $\psi:X \rightarrow \sP(\fA)$ be a weak*-continuous function. Let $x_0 \in X$ and suppose $\psi_{x_0}$ is excised by a decreasing sequence of projections $(P_n)_{n \in \bbN}$ in $\fA$ such that $\psi_{x_0}(P_n) = 1$ for all $n \in \bbN$ and such that the range $\pi_x(P_n)\cH_x$ is infinite-dimensional for all $x \in X$ and $n \in \bbN$, where $(\cH_x, \pi_x)$ is the GNS representation of $\psi_x$. Then there exists a continuous map $V:X \times [0,1) \rightarrow \Unitary(\fA)$ such that 
\begin{enumerate}
	\item $V(x,0) = \1$ for all $x \in X$,
	\item $V(x_0, t) = \1$ for all $t \in [0,1)$,
	\item the map $H:X \times I \rightarrow \sP(\fA)$ defined by
	\begin{equation}\label{eq:H_nulhomotopy}
	H(x,t) = \begin{cases} V(x,t)\psi_x &\tn{for } t < 1 \\ \psi_{x_0} &\tn{for } t = 1 \end{cases}
	\end{equation}
	is a weak*-continuous nulhomotopy of $\psi$ relative to $x_0$.
\end{enumerate}
\end{thm}

\begin{proof}
Applying Theorem \ref{thm:Michael_application} with $P = P_1$, we get a $\Unitary(\fA)$-homotopy $U_1:X \times I \rightarrow \Unitary(\fA)$ relative to $\qty{x_0}$ such that \eqref{eq:end_of_homotopy} holds for all $x \in X$. Define $\psi^1:X \rightarrow \sP(\fA)$ by $\psi_x^1 = U_1(x,1)\psi_x$ for all $x \in X$. Then $\psi^1$ is weak*-continuous, $\psi_{x_0}^1 = \psi_{x_0}$, and $\psi^1_x(P_1) = 1$ for all $x \in X$.

Set $\psi^0 \defeq \psi$ and $P_0 = \1$. Suppose for some $n \in \bbN$ we have constructed weak*-continuous maps $\psi^1,\ldots, \psi^n:X \rightarrow \sP(\fA)$ such that $\psi^i_{x_0} = \psi_{x_0}^0$ and $\psi^i_x(P_i) = 1$ for all $i \leq n$ and $x \in X$, and suppose further that for each $1 \leq i \leq n$ we have constructed a $\Unitary(\fA)$-homotopy $U_i: X \times I \rightarrow \Unitary(\fA)$ relative to $\qty{x_0}$ from $\psi^{i-1}$ to $\psi^i$ such that 
\begin{equation}\label{eq:U(A)_homotopy_matrix}
U_i(x,t) = P_{i-1}U_i(x,t)P_{i-1} + \1 - P_{i-1}
\end{equation}
for all $1 \leq i \leq n$ and all $(x,t) \in X \times I$. We construct $\psi^{n+1}$ and $U_{n+1}$.

Since $\psi^n_x(P_n) = 1$, we know by Lemma \ref{lem:compressed_rep} \ref{ite:c} that $\psi^n_x|_{P_n\fA P_n}$ is pure for all $x \in X$.  Moreover, the map $x \mapsto \psi^n_x|_{P_n\fA P_n}$ is weak*-continuous since $\psi^n_x$ is weak*-continuous. Since the sequence $(P_n)_{n \in \bbN}$ is decreasing, we know $P_{n+1} = P_n P_{n+1} P_n \in P_n \fA P_n$. Fix $x \in X$ and let $(\cK, \rho)$ and $(\cL, \sigma)$ be the GNS representations of $\psi^n_x|_{P_n \fA P_n}$ and $\psi^n_x$, respectively. We claim that $\rho(P_{n+1})\cK$ is infinite-dimensional. By Lemma \ref{lem:compressed_rep} \ref{ite:c} there exists a unitary $W_1:\cK \rightarrow \sigma(P_n)\cL$ intertwining $\rho$ and $\sigma_{P_n}$. Since $\psi^n_x$ is unitarily equivalent to $\psi_x$, there exists a unitary $W_2:\cL \rightarrow \cH_x$ intertwining $\sigma$ and $\pi_x$. It is then easily seen that $W_2^*\pi_x(P_{n+1})\cH_x \subset \sigma(P_n)\cL$ and $\rho(P_{n+1})\cK = W_1^*W_2^*\pi_x(P_{n+1})\cH_x$. Thus, since $\pi_x(P_{n+1})\cH_x$ is infinite-dimensional, so is $\rho(P_{n+1})\cK$.

By Theorem \ref{thm:Michael_application}, there exists a $\Unitary(P_n\fA P_n)$-homotopy $\tilde U_{n+1}:X \times I \rightarrow \Unitary(P_n\fA P_n)$ relative to $\qty{x_0}$ such that 
\[
\psi^n_x|_{P_n\fA P_n}(\tilde U_{n+1}(x,1)^*P_{n+1}\tilde U_{n+1}(x,1)) = 1
\]
for all $x \in X$. Define $U_{n+1}:X \times I \rightarrow \Unitary(\fA)$ by $U_{n+1}(x, t) = \tilde U_{n+1}(x,t) + \1 - P_n$ and define $\psi^{n+1}:X \rightarrow \sP(\fA)$ by $\psi^{n+1}_x = U_{n+1}(x,1)\psi^n_x$. Note that $U_{n+1}$ is a $\Unitary(\fA)$-homotopy relative to $\qty{x_0}$ from $\psi^n$ to $\psi^{n+1}$. Furthermore, $\psi^{n+1}_{x_0} = \psi^0_{x_0}$ and since $(\1 - P_n)P_{n+1} = P_{n+1}(\1 - P_n) = 0$, we have $\psi_x^{n+1}(P_{n+1}) = 1$ for all $x \in X$.

By recursion we obtain for each $i \in \bbN$ a weak*-continuous $\psi^i:X \rightarrow \sP(\fA)$ such that $\psi^i_{x_0} = \psi_{x_0}$ and $\psi^i_x(P_i) = 1$ for all $x \in X$, and we also obtain a $\Unitary(\fA)$-homotopy $U_i:X \times I \rightarrow \Unitary(\fA)$ relative to $\qty{x_0}$ from $\psi^{i-1}$ to $\psi^i$, such that \eqref{eq:U(A)_homotopy_matrix} holds for all $i$, $x$, and $t$.

We now define $V: X \times [0,1) \to  \Unitary(\fA)$ as follows. Given $i \in \bbN$, define $V$ on $X \times [1 - 2^{-i+1}, 1-2^{-i}]$ as
\[
V({x,t}) = U_i(x,s) U_{i-1}({x,1}) U_{i-2}({x,1}) \cdots U_1({x,1}), \quad \tn{ where $s = \frac{t - 1 + 2^{-i+1}}{2^{-i+1} - 2^{-i}}$}.
\]
Since $U_i({x,0}) = \1$ for all $x$ and $i$, we see that 
\begin{align*}
V({x,1-2^{-i+1}})= U_{i-1}({x,1})U_{i-2}({x,1})\cdots U_1({x,1}).
\end{align*}
Note also that
\[
V({x,1-2^{-i}}) = U_i({x,1})U_{i-1}({x,1}) \cdots U_1({x,1}).
\]
Thus, the functions $V: X \times [1-2^{-i+1}, 1-2^{-i}] \to  \fA$ glue together to a continuous function $V: X \times [0,1) \to  \fA$. Note that $V({x,0}) = \1$ for all $x \in X$ and $V({x_0, t}) = \1$ for all $t \in I$.

Finally, we define $H : X \times I \to  \sP(\fA)$ as in \eqref{eq:H_nulhomotopy}.
Since $V({x,0}) = \1$ for all $x$, we see that $H(x, 0) = \psi_x$. Since $V(x_0, t) = \1$ for all $t \in [0,1)$,  we see that $H(x_0,t) = \psi_{x_0}$ for all $t \in I$. Clearly $H(x,1)$ is constant at $\psi_{x_0}$. All that remains is to show weak*-continuity of $H$.

Let $B \in \fA$. On $X \times [0,1)$, we have $H(x,t)(B) = (V(x,t)\psi_x)(B)$, which is continuous since $V$ is continuous and $\psi_x$ is weak*-continuous. Thus it remains to show continuity of $(x,t) \mapsto H(x, t)(B)$ at an arbitrary point $(y,1)$ with $y \in X$. Fix $\varepsilon > 0$. Observe that for any $x \in X$, any $n, i \in \bbN$ with $i > n + 1$ and any $t \in [1-2^{-i+1}, 1-2^i]$, we have
\begin{align*}
H(x,t)(P_n) &= (V({x,t}) \psi_x)(P_n)\\
&= \qty[\qty(U_i({x,s}) U_{i-1}({x,1})\cdots U_{n+1}({x,1})) \cdot \qty(U_n({x,1}) \cdots U_1({x,1})) \psi_x](P_n)\\
&= \qty[\qty(U_i({x,s})  U_{i-1}({x,1}) \cdots U_{n+1}({x,1})) \cdot \psi^n_x](P_n) \\
&= \psi^n_x(P_n) = \psi_{x_0}(P_n)=1,
\end{align*}
where we have used the fact that $[U_i(x,t), P_n] = 0$ for all $i > n$ and all $(x,t) \in X \times I$, as follows from \eqref{eq:U(A)_homotopy_matrix} and the fact that the sequence $(P_n)_{n \in \bbN}$ is decreasing. Choose $n \in \bbN$ such that $\norm{P_nBP_n - \psi_{x_0}(B)P_n} < \varepsilon$ and set $i = n+2$. Then for any $(x,t) \in X \times [1 - 2^{-i+1}, 1]$, we have
\begin{align*}
\abs{H(x,t)(B) - H(y,1)(B)} &= \abs{H(x,t)(P_nBP_n) - \psi_{x_0}(B)}\\
&= \abs{H(x,t)(P_nBP_n) - H(x,t)(\psi_{x_0}(B)P_n)}\\
&\leq \norm{P_nBP_n - \psi_{x_0}(B)P_n} < \varepsilon.
\end{align*}
This proves continuity at $(y,1)$, completing the proof.
\end{proof}

Taking $X$ to be an $n$-sphere yields the following corollary, giving weak contractibility of the set of pure states with the weak* topology. By passing to the unitization, we are able to deal with non-unital $C^*$-algebras as well, with the exception of the compact operators on an infinite-dimensional Hilbert space. Recall that a $C^*$-algebra is called \textit{elementary} if it is $*$-isomorphic to the compact operators on some Hilbert space and \textit{nonelementary} otherwise.

\begin{cor}\label{cor:final_result}
If $\fA$ is a nonelementary, separable, simple $C^*$-algebra with real rank zero, then $\sP(\fA)$ with the weak* topology is weakly contractible, i.e, all homotopy groups of $\sP(\fA)$ are trivial.
\end{cor}

\begin{proof}
By \cite[Thm.~1.7, Thm.~5.6, \& Thm.~5.9]{EilCCCA} or \cite[Thm.~1.21]{Spiegel}, we know $\sP(\fA)$ with the weak* topology is path-connected, so the homotopy groups do not depend on the choice of basepoint. Let $\psi_0 \in \sP(\fA)$ be an arbitrary point. Let $n \in \bbN$ and let $\psi:\bbS^n \rightarrow \sP(\fA)$ be a weak*-continuous map with $\psi_{x_0} = \psi_0$, where $x_0$ is the basepoint of $\bbS^n$.

Suppose $\fA$ is unital. Setting $\fN = \qty{A \in \fA: \psi_{0}(A^*A) = 0}$, we know $\fN \cap \fN^*$ is $\sigma$-unital since it is separable, so by Proposition \ref{prop:excising_sequence} $\psi_{x_0}$ can be excised by a decreasing sequence of projections $(P_n)_{n \in \bbN}$ such that $\psi_{x_0}(P_n) = 1$ for all $n \in \bbN$. If $(\cH_x, \pi_x)$ is the GNS representation of $\psi_x$, then the range $\pi_x(P_n)\cH_x$ is infinite-dimensional for all $x \in \bbS^n$ since $\fA$ is simple, separable, and infinite-dimensional (cf.~\cite[Thm.~2.4.9]{MurphyCAOT}). Theorem \ref{thm:iteration_nulhomotopy} now implies that $\psi$ is nulhomotopic relative to $x_0$.

Suppose $\fA$ is non-unital and let $\iota:\fA \rightarrow \tilde \fA$ be the embedding of $\fA$ in its unitization. Each pure state on $\fA$ has a unique extension to a pure state on $\tilde \fA$, and the map taking a pure state of $\fA$ to its extension is weak*-continuous, so we obtain a weak*-continuous map $\tilde \psi:\bbS^n \rightarrow \sP(\fA)$ such that $\tilde \psi_x \circ \iota = \psi_x$ for all $x \in \bbS^n$. As in the previous paragraph, since $\tilde \fA$ is unital, separable, and real rank zero, Proposition \ref{prop:excising_sequence} implies $\tilde \psi_{x_0}$ can be excised by a decreasing sequence of projections $(P_n)_{n \in \bbN}$ such that $\tilde \psi_{x_0}(P_n) = 1$ for all $n \in \bbN$. Let $(\cH_x',  \pi_x')$ be the GNS representation of $\tilde \psi_{x}$ and note that $(\cH_x', \pi_x')$ is unitarily equivalent to $(\cH_x, \tilde \pi_x)$, where $(\cH_x, \pi_x)$ is the GNS representation of $\psi_x$ and $\tilde \pi_x$ is the unique extension of $\pi_x$ to a unital $*$-homomorphism $\tilde \fA \rightarrow \B(\cH_x)$.

Suppose $\pi_x'(P_n)\cH_x'$ is finite-dimensional for some $x \in X$ and $n \in \bbN$. Then so is $\tilde \pi_x(P_n)\cH_x$ by unitary equivalence. By the Kadison transitivity theorem, there exists $A \in \fA$ such that $\pi_x(A)\Omega_x = \Omega_x$, where $\Omega_x \in \cH_x$ is the distinguished cyclic vector representing $\psi_x$. Since $\Omega_x$ also represents $\tilde \psi_x$ and $\tilde \psi_x(P_n) = 1$, we know $\tilde \pi_x(P_n)\Omega_x = 1$. Then $P_n \iota(A) = \iota(B)$ for some $B \in \fA$ and $\tilde \pi_x(P_n \iota(A))\Omega_x = \Omega_x$, hence $\pi_x(B)$ is a nonzero finite-rank operator. By \cite[Thm.~2.4.9]{MurphyCAOT}, this implies $\cK(\cH_x) \subset \pi_x(\fA)$, hence $\cK(\cH_x) = \pi_x(\fA)$ by simplicity of $\pi_x(\fA)$. But this contradicts that $\fA$ is nonelementary. Thus, $\pi_x'(P_n)\cH_x'$ is infinite-dimensional for all $x \in X$ and $n \in \bbN$. 

With respect to $\tilde \psi$, we now obtain a continuous map $V:\bbS^n \times [0,1) \rightarrow \Unitary(\tilde \fA)$ as in Theorem \ref{thm:iteration_nulhomotopy}, with associated nulhomotopy $H$. The functionals $(V(x,t)\tilde \psi_x) \circ \iota$ are pure states on $\fA$ since they are represented by the irreducible representation and unit vector $(\cH_x, \pi_x, \tilde \pi_x(V(x,t))\Omega_x)$. Thus, $\bbS^n \times I \rightarrow \sP(\fA)$, $(x,t) \mapsto H(x,t) \circ \iota$ yields a well-defined weak*-continuous nulhomotopy of $\psi$, as desired.
\begin{comment}
Finally, suppose $\fA = \cK(\cH)$ for some separable infinite-dimensional Hilbert space $\cH$. There exists $\Omega_0 \in \bbS\cH$ such that $\psi_{x_0}(A) = \ev{\Omega_0, A\Omega_0}$ for all $A \in \cK(\cH)$. Extend $\Omega_0$ to a basis $(\Omega_n)_{n \in \qty{0} \cup \bbN}$ of $\cH$. Given $n \in \bbN$, define $Q_n = \sum_{i=1}^n \ketbra{\Omega_i}$ and $P_n = \1 - \iota(Q_n)$. Clearly $\tilde \psi_{x_0}(P_n) = 1$ for all $n$ and we claim that $(P_n)_{n \in \bbN}$ excises $\tilde \psi_{x_0}$.
\end{comment}
\end{proof}

The case when $\fA = \cK(\cH)$, the set of compact operators on an infinite-dimensional Hilbert space $\cH$, may be easily synthesized from well-known results. We do not need to restrict to separable Hilbert spaces.

\begin{prop}\label{prop:K(Z,2)}
If $\fA$ is infinite-dimensional and elementary, then the weak* topology on $\sP(\fA)$ coincides with the norm topology on $\sP(\fA)$ and $\sP(\fA)$ is an Eilenberg-MacLane space of type $K(\bbZ,2)$.
\end{prop}

\begin{proof}
There exists an infinite-dimensional Hilbert space $\cH$ such that $\fA$ is $*$-isomorphic to $\cK(\cH)$, the set of compact operators on $\cH$. It is known that $\sP(\cK(\cH))$ consists precisely of the vector states $\omega(A) = \ev{\Omega, A\Omega}$ for unit vectors $\Omega \in \bbS \cH$  (see e.g., \cite[Example~5.1.1]{MurphyCAOT}) and that, with the norm topology, $\sP(\cK(\cH))$ is homeomorphic to the projective Hilbert space $\bbP \cH$, endowed with the quotient of the norm topology on $\bbS \cH$ (see e.g., \cite[Cor.~2.2]{Spiegel}). The projective Hilbert space $\bbP \cH$ is a $K(\bbZ, 2)$, as follows from the contractibility of $\bbS \cH$, the fact that $\bbS \cH$ is a fiber bundle over $\bbP \cH$ with typical fiber $\Unitary(1)$ (see e.g., \cite[Thm.~2.6]{Spiegel}), and the long exact sequence on fiber bundles \cite[Sec.~17]{Steenrod}. 

All that remains to do is show that the weak* topology on $\sP(\cK(\cH))$ coincides with the norm topology. Fix $\omega \in \sP(\cK(\cH))$ and let $\Omega \in \bbS\cH$ such that $\omega(A) = \ev{\Omega, A\Omega}$. Given $\varepsilon > 0$, if $\psi \in \sP(\cK(\cH))$ is represented by $\Psi \in \bbS \cH$ and $\psi(\ketbra{\Omega}) = \abs{\ev{\Psi, \Omega}}^2 > 1 - \varepsilon$, then by \cite[Prop.~4.6]{RobertsRoepstorffSBCAQT} we have
\[
\norm{\psi - \omega} = 2\sqrt{1 - \abs{\ev{\Psi, \Omega}}^2} < 2 \sqrt{\varepsilon}.
\]
We see that $\qty{\psi \in \sP(\cK(\cH)): \psi(\ketbra{\Omega}) > 1 - \varepsilon}$ is a neighborhood of $\omega$ in the weak* topology contained in the open ball centered on $\omega$ of radius $2\sqrt{\varepsilon}$. Since $\omega$ and $\varepsilon$ were arbitrary, this proves that the norm topology on $\sP(\cK(\cH))$ is coarser than the weak* topology, and the converse is trivial.
\end{proof}

%!TEX root = pure_state_homotopy.tex

\section{Comparison with the Cases of Commutative and Rotation Algebras}\label{sec:comparison}

In the case of a commutative unital $C^*$-algebra $\fA$, the space of pure states $\sP(\fA)$ coincides
 with the Gelfand spectrum $\sigma (\fA)$ defined as the space of characters on $\fA$ endowed with the
weak$^*$-topology.  By the commutative Gelfand-Naimark theorem, the functor $\sigma$ provides an
equivalence between the category of commutative  $C^*$-algebras and the category of locally compact Hausdorff spaces.
Its  left (and right) adjoint is given by the functor $C_0$ which associates to a locally compact Hausdorff space $X$ the
$C^*$-algebra of continuous complex-valued functions on $X$ vanishing at infinity. 
This implies that for a locally compact Hausdorff space $X$ the $k$-the homotopy group of the state space
$ \sP (C_0(X))$ coincides with $\pi_k (X)$. 
%In other words this means that the homotopy types of every (locally) compact topological spaces can appear as the
%homotopy types of a commutative $C^*$-algebra.
Hence  the homotopy type of every locally compact Hausdorff space can appear as the homotopy type of the
pure state space of a commutative $C^*$-algebra. This is quite the opposite to the case of nonelementary separable simple $C^*$-algebras $\fA$
of real rank zero which are necessarily  noncommutative and all have weakly contractible $\sP(\fA)$
by our main result. 

One can also see that the homotopy type of the space of pure states is not an invariant under
deformation or of continuous fields of $C^*$-algebras. We explain this in more detail by the example of noncommutative tori. 
%Consider the family of noncommutative tori $\big( \fA_\theta \big)_{\theta \in \R}$.
Recall, for example from \cite{RieffelNCT}, \cite{ConnesNCG}, \cite{DavidsonCStarExample} or \cite{VarillyEtAl},
that for $\theta \in \R$ the \emph{noncommutative torus} or in other terminology the \emph{rotation algebra} $\fA_\theta$ is defined as the universal unital $C^*$-algebra generated by two unitary elements $U$ and $V$
fulfilling the commutation relation $V U = e^{-2\pi i \theta} UV$.
For irrational $\theta$ one way to construct $\fA_\theta$ is as the $C^*$-subalgebra of $\B (L^2 (\bbS^1))$
generated by the two unitary operators
%$U,V \in \B ( L^2 (\bbS^1))$ given by
\begin{align}
  U &: \B ( L^2 (\bbS^1))\to \B ( L^2 (\bbS^1)), \quad Uf(z) = zf(z) \ , \\
  V &: \B ( L^2 (\bbS^1))\to \B ( L^2 (\bbS^1)), \quad Vf(z) = f(ze^{-2\pi i\theta}) \ . 
\end{align}
%\[ Uf(z) = zf(z) \quad \text{and} \quad Vf(z) = f(ze^{-2\pi i\theta})\quad \text{for }
%  f\in L^2 (\bbS^1), \: z \in \bbS^1 \ .
%\]
Alternatively, for all real $\theta$, $\fA_\theta$ can be identified with the
crossed product $C^*$-algebra 
$C (\bbS^1) \rtimes_{\tau_\theta} \Z$ defined by the action
\[
  \tau_\theta : \: \Z \to \operatorname{Aut} \big(C (\bbS^1)\big), \quad 
  \tau(n) (f) (z) = f ( e^{-2\pi n i \theta} z) \ . 
\]
The family  $\big( \fA_\theta \big)_{\theta \in \R}$ then forms a continuous field of $C^*$-algebras in the sense
of Dixmier \cite[Chpt.\ 10]{Dixmier} and can be even understood as a strict or in other words $C^*$-algebraic
deformation quantization of the $2$-torus $\T^2$; see e.g.\ \cite{RieffelNCT}.
In particular this means that $\fA_0$ coincides with the commutative $C^*$-algebra $C (\T^2)$.

The irrational rotation algebras -- which by definition are the  $C^*$-algebras $\fA_\theta$ with
$\theta$ irrational -- are simple \cite[Thm.\ VI.1.4]{DavidsonCStarExample}. Moreover,
by construction it is clear that all such $\fA_\theta$ are separable and infinite-dimensional.
By unitality they are also non-elementary.  
Finally, one knows by \cite[Thm.\ 1.5]{BlackadarKumjianRordam} that the irrational rotation algebras have real
rank $0$. 
By our main result Corollary \ref{cor:final_result}, the homotopy groups $\pi_k (\sP (\fA_\theta ))$ vanish for
irrational $\theta$ and natural $k$ whereas
\[
  \pi_k (\sP (\fA_0)) =
  \begin{cases}
    \Z^2 &\text{for } k =1 \ , \\
    \{ 0 \}& \text{else} \ .
  \end{cases}
\]
This shows that unlike in the case of a locally trivial bundle of $C^*$-algebras the homotopy groups of the pure state
spaces of noncommutative tori are not invariant  along the deformation parameter $\theta$, at least not around around
$\theta =0$.

The situation becomes even ``wilder'' when looking at rational rotation algebras  $\fA_\theta$ for non-integral
values of $\theta$.
So let $\theta = \frac pq$ with $p$ and $q$ divisor free non-zero integers and $q\geq 2$. It is then well-known -- see e.g.\ Example 8.46 in D. Williams' book \cite{WilliamsCrossedProducts} -- that  the spectrum $\hat{\fA}_\theta$ and the
set of primitive ideals $\operatorname{Prim} (\fA_\theta)$ endowed with the hull-kernel topology can be naturally identified
with the 2-torus $\T^2$. Moreover, the function $\lambda : \hat{\fA}_\theta  \to \operatorname{Prim} (\fA_\theta)$
which maps  the unitary equivalence class $[\pi]$ of an irreducible representation $\pi : \fA_\theta \to \B(\cH_\pi)$
to the kernel $\ker \pi$ is a homeomorphism by \cite[3.1.6.\ Prop.]{Dixmier}.
Finally, $\fA_\theta$ is a $q$-homogeneous $C^*$-algebra meaning that all of its
irreducible non-trivial representations act on $\C^q$; see again \cite[Ex.\ 8.46]{WilliamsCrossedProducts}.
By Section 3.2 of Fell's paper \cite{Fell}, $\fA_\theta$ is isomorphic to the $C^*$-algebra
$\Gamma (B_\theta)$ of continuous sections of a locally trivial $\operatorname{Mat}_q(\C)$-bundle $\pi^{B_\theta} :  B_\theta \to \T^2$.
In particular, this is a $C^*$-algebra bundle in the sense of
\cite[Sec.\ 4.1]{Spiegel}. Since the pure state space of $\operatorname{Mat}_q(\C)$ is given by the projective space $\C\bbP^{q-1}$,
forming fiberwise the pure state space in the bundle $\pi^{B_\theta} : B_\theta \to \T^2$ gives rise to a locally
trivial $\C\bbP^{q-1}$-bundle $\pi^{P_\theta}: P_\theta \to \T^2$ by \cite[Sec.\ 4.1]{Spiegel} .

\begin{lem}
  With assumptions and notation as above, the canonical map 
  \[
    \Pi : \:  P_\theta \to \sP (\Gamma ( B_\theta  )), \quad  \omega \mapsto
    \Big( \Gamma ( B_\theta  ) \ni s \mapsto \omega \big( s\big(\pi^{P_\theta}(\omega) \big)\big)  \in \C \Big)
  \]
  is a homeomorphism, where the space of pure states $\sP (\Gamma ( B_\theta  ))$
  is endowed with the weak$^*$-topology.
\end{lem}
\begin{proof}
  Before proving that $\Pi$ is a homeomorphism we show that $\Pi$ is well-defined meaning that
  for every element $\omega \in P_\theta$ the image $\Pi(\omega)$ is a pure state of $\Gamma (B_\theta)$.
  It is clear that $\Pi(\omega)$ is a state, so it remains to verify that $\Pi(\omega)$ is pure. 
  To this end let $z= \pi^{P_\theta} (\omega)$ be the footpoint of $\omega$,
  $\mathfrak{N}_\omega = \{ A \in (B_\theta)_z \mid \omega (A^*A)=0\}$ the Gelfand ideal
  of $\omega$, and $\mathfrak{N}_{\Pi(\omega)} = \{ s \in \Gamma (B_\theta) \mid \Pi(\omega) ( s^*s ) = 0 \}$
  the Gelfand ideal of $\Pi(\omega)$. By definition of $\Pi$, the left ideal $\mathfrak{N}_{\Pi(\omega)}$ coincides
  with the left ideal $\{ s \in \Gamma (B_\theta) \mid s(z) \in \mathfrak{N}_\omega \}$.
  Hence evaluation at the footpoint $z$ provides an isomorphism between GNS Hilbert spaces
  \[
    \cH_{\Pi(\omega)} :=  \Gamma (B_\theta) / \mathfrak{N}_{\Pi(\omega)} \to \cH_\omega := (B_\theta)_z/\mathfrak{N}_\omega, \quad
    s + \mathfrak{N}_{\Pi(\omega)} \mapsto s(z) + \mathfrak{N}_\omega \ .
  \]
  Note that $\cH_{\Pi(\omega)}$ and $\cH_\omega$ are finite-dimensional and that the isomorphism above commutes with the GNS representations. 
  Since the action of $(B_\theta)_z$ on $\cH_\omega$ is irreducible one concludes that the GNS representation of
  $\Gamma (B_\theta)$ on $\cH_{\Pi(\omega)}$ is irreducible as well, so $\Pi(\omega)$ is pure. 
    
  Injectivity of $\Pi$ follows from the observation that for each base point $z\in \bbS^1$ the fiber $(B_\theta)_z$ separates
  points of the fiber $(P_\theta)_z$. The map  $\Pi$ is continuous since for every $s \in \Gamma ( B_\theta )$ the evaluation map 
  $P_\theta \to \C$, $\omega \mapsto \Pi(\omega) (s)$  is continuous.
  Observe that  $P_\theta$ is compact because both the base space and the typical fiber of $P_\theta$ are compact.
  Hence  $\Pi$ is a closed map.

  It remains to show that $\Pi$ is surjective. To this end we need to recall a few
  well-known facts. First, denoting by $\hat{\Gamma} ( B_\theta  )$ the spectrum of
  $\Gamma ( B_\theta  )$, there exists  a canonical map 
  $\kappa : \sP (\Gamma ( B_\theta  )) \to \hat{\Gamma} ( B_\theta  )$
  % , \quad \varrho \mapsto [\pi_\varrho]
  which associates to every pure state $\varrho$ the unitary equivalence class
  $[\pi_\varrho]$ of the GNS representation of $\varrho$.
  The map $\kappa$ is continuous, open  and surjective by \cite[Thm.\ 3.4.11]{Dixmier}.
  Second, as already pointed out above, the natural map 
  $\lambda : \hat{\fA}_\theta \to \operatorname{Prim} (\fA_\theta)$, $[\pi] \mapsto \ker \pi$
  is a homeomorphism.  Third, the center of $\Gamma ( B_\theta  )$ is given by the $C^*$-subalgebra
  $C (\T^2) \1$, and the map
  \[
    \begin{split}
      \mu : \:  \operatorname{Prim} (\Gamma ( B_\theta  )) & \to
      \operatorname{Prim} \big( C (\T^2) \big) \cong \sigma \big( C (\T^2) \big) \cong \T^2 , \\
        J  & \mapsto \mathfrak{m}_J := \{ f \in C (\T^2) \mid f \Gamma ( B_\theta  ) \subset J \} \cong J \cap C (\T^2) \1
    \end{split}
  \]
  is a homeomorphism as well. This is a consequence of a variant of the Dauns-Hofmann theorem \cite[3.4.~Thm.]{HofmannToposym2}
  % \cite[Cor.~4.4.8]{Pedersen}
  which tells that the pullback by $\mu$ provides an isomorphism between the center of $\Gamma ( B_\theta  )$
  and the $C^*$-algebra of continuous functions on $\operatorname{Prim} (\Gamma ( B_\theta  ))$.
  An explicit proof that $\mu$ is continuous and surjective can be found in Lemma 8.10 of
  the paper \cite{DaunsHofmannMemoirsAMS} which also contains the original statement of the Dauns-Hofmann theorem. 
  % by \cite[Ex.\ 8.46]{WilliamsCrossedProducts}. %well-defined and an ismorphism.
  Note that in the definition of $\mu$ we used that the spectrum $\sigma \big( C (\T^2)\big)$ coincides with the maximal ideal space of $C (\T^2)$ 
  and that for a primitive ideal $ J \subset \Gamma ( B_\theta  )$ the set $\mathfrak{m}_J$ 
  % $\{ f \in C (\T^2) \mid f \Gamma ( B_\theta  ) \subset I \} $
  is a maximal ideal in $C (\T^2)$. 
  % , where $\sigma (C (\T^2))$ is identifie with the maximal spectrum of  $C (\T^2)$.
  
  Now let $\varrho \in \sP (\Gamma ( B_\theta  ))$,  $z = \mu \lambda \kappa (\varrho) \in \T^2$
  and $I_\varrho \subset \Gamma ( B_\theta  )$ be the kernel of $\pi_\varrho$. Then
  $I_\varrho = \mu^{-1} (z) = \mathfrak{m}_z \, \Gamma ( B_\theta  )$, where $\mathfrak{m}_z\subset C (\T^2)$ is the maximal
  ideal of functions vanishing at $z$. Therefore, the quotient algebra $\Gamma ( B_\theta  )/ I_\varrho$ is isomorphic to
  $(B_\theta)_z$, and the state
  $\varrho$ factors through a pure state $\overline{\varrho} :  (B_\theta)_z \to \C$.
  Hence  $\overline{\varrho} \in (P_\theta)_z$ and $\Pi (\overline{\varrho})  = \varrho$ by construction.
  This proves surjectivity of $\Pi$.
\end{proof}

%\begin{rem}
%  Note that $P_\theta$ is a compact finite dimensional topological manifold, which implies that the norm topology
%  and the weak$^*$ topology on the space of pure states $\sP (\Gamma ( B_\theta  )) \cong P_\theta$ coincide.
%\end{rem}

Now we can compute the homotopy groups of $\sP (\Gamma ( B_\theta  ))$. By the lemma,
these coincide with the homotopy groups of the total space $P_\theta$. The long exact homotopy sequence of the
fiber bundle $P_\theta \to \T^2$ is given by
\[
  \ldots
  \to \pi_{k+1} (\T^2) \to \pi_k (\C\bbP^{q-1} ) \to \pi_k ( P_\theta ) \to \pi_k (\T^2) \to
  \pi_{k-1}  (\C\bbP^{q-1} ) \to \ldots \ .
\]
Because $\C\bbP^{q-1}$ is connected and simply connected one obtains for $\theta = \frac pq$
\[
  \pi_k (\sP (\fA_\theta)) = \pi_k (P_\theta) =
  \begin{cases}
    \{ 0\} &\text{for } k =0 \ , \\
    \Z^2 &\text{for } k =1 \ , \\
    \pi_k (\C\bbP^{q-1} ) & \text{for } k \geq 2 \ .
  \end{cases}
\]
By the long exact homotopy sequence of the fiber bundle $\bbS^{2q-1} \to \C\bbP^{q-1}$, the homotopy groups
$\pi_k (\C\bbP^{q-1})$ can be described in terms of the homotopy groups of the spheres.
A straightforward computation gives
\[
  \pi_k (\C\bbP^{q-1}) =
  \begin{cases}
    \{ 0\} &\text{for } k =0,1 \ , \\
    \Z &\text{for } k =2 \ , \\
    \pi_k (\bbS^{2q-1} ) & \text{for } k > 2 \ .
  \end{cases}
\]
This entails, still for $\theta = \frac pq$, 
\[
  \pi_k (\sP (\fA_\theta)) = \pi_k (P_\theta) = 
  \begin{cases}
    \Z &\text{for } k =2 \ , \\
    \{ 0\} &\text{for } 2 < k < 2q-1 \ , \\
    \Z  &\text{for } k = 2q-1 \ , \\
    \pi_k (\bbS^{2q-1} ) & \text{for } k > 2q-1 \ ,
  \end{cases}
\]
and indicates the change of the homotopy groups $\pi_k (\sP (\fA_\theta)) = \pi_k (\C\bbP^{q-1})$
along the parameter $\theta = \frac pq$. In particular one concludes that the family of pure state spaces
$\big(\sP (\fA_\theta)\big)_{\theta \in \R}$ is nowhere locally trivial. 
By \cite[3.1.~Thm.]{RieffelPLMS1983}, all rational rotation algebras are Morita equivalent to the commutative
$C^*$-algebra $C(\T^2)$, but their state spaces have homotopy types different from the torus. 
In general, the homotopy type of the pure state space of a $C^*$-algebra is therefore not Morita invariant.

%!TEX root = pure_state_homotopy.tex

\appendix
\section{Contractibility of the Infinite-Dimensional Unit Sphere}
\label{sec:contractible_sphere}

It has been known for many decades that the unit sphere of a separable infinite-dimensional Hilbert space is contractible \cite{Kakutani_SH_contractible}. The same result holds without the condition of separability by very nearly the same proof, but we are unfortunately unable to find this non-separable version in the literature. We provide this short appendix to show how to handle the non-separable case.

\begin{thm}
Let $\cH$ be an infinite-dimensional Hilbert space. The unit sphere $\bbS \cH$ is contractible.
\end{thm}

\begin{proof}
Let $\cE$ be an orthonormal basis for $\cH$. By a simple application of Zorn's lemma, there exists a partition of $\cE = \bigcup_i \cE^i$ into pairwise disjoint countably infinite subsets $\cE^i \subset \cE$. For each index $i$, let $\qty{\Psi_1^i, \ldots, \Psi_n^i,\ldots}$ enumerate the elements of $\cE^i$. Define an isometry $U:\cH \rightarrow \cH$ by setting $U(\Psi^i_n) = \Psi^i_{n+1}$ for all $i$ and $n$ and then extending $U$ to all of $\cH$ by linearity and continuity.

We claim $U$ has no eigenvectors. Suppose to the contrary that $\Omega \in \bbS \cH$ is an eigenvector of $U$ with eigenvalue $\lambda$. Given an index $i$, define observe that $\ev{\Psi^i_1, U\Omega} = 0$ by definition of $U$, hence $\ev{\Psi^i_1, \Omega} = \ev{\Psi^i_1, \lambda^{-1}U\Omega} = 0$. If $\ev{\Psi^i_n, \Omega} = 0$ for some $n \in \bbN$, then since $U$ is an isometry we have
\[
\ev{\Psi^i_{n+1}, \Omega} = \ev{U\Psi^i_{n}, \lambda^{-1}U\Omega} = \ev{\Psi^i_n, \lambda^{-1}\Omega} = 0.
\]
By induction we have $\ev{\Psi^i_n, \Omega} = 0$ for all $n$. Since $i$ was arbitrary, we see that $\Omega = 0$, a contradiction.

We define a homotopy $H:\bbS \hilbH \times I \rightarrow \bbS \hilbH$ by
\[
H(\Omega,t) = \frac{tU\Omega + (1 - t)\Omega}{\norm{tU\Omega + (1 - t)\Omega}}.
\]
Note that the denominator is never zero---this is clear when $t = 0$ and when $t = 1$, and for all other $t$ it follows from the fact that $U$ has no eigenvectors. This shows that the identity map on $\bbS \cH$ is homotopic to the restriction $U:\bbS \hilbH \rightarrow \bbS \hilbH$. 

Fix $i \in I$. Define a second homotopy $H':\bbS \hilbH \times I \rightarrow \bbS \hilbH$ by
\[
H'(\Omega,t) = \frac{t\Psi^i_1 + (1 - t)U\Omega}{\norm{t\Psi^i_1 + (1 - t)U\Omega}}.
\]
Again we note that the denominator is never zero---this is clear when $t = 0$ and when $t = 1$, and for all other $t$ it follows from the fact that $\ev{\Psi^i_1, U\Omega} = 0$ for all $\Omega \in \hilbH$. This gives a homotopy of the restriction $U:\bbS \hilbH \rightarrow \bbS \hilbH$ with the constant map at $\Psi_1^i$. 
\end{proof}

\bibliographystyle{amsalpha}
\bibliography{AQMref}

\end{document}